\documentclass [11pt,a4paper]{article}
\usepackage[T1]{fontenc}
\usepackage[ansinew]{inputenc}
\usepackage{amsthm,amsmath,amssymb}
\usepackage{graphicx}
\usepackage{cite}
\usepackage{color}

\newcommand{\Optime}{\tau_{*}}
\newcommand{\dn}{\partial_{\bf n}}

\newcommand{\N}{\mathbb{N}}

\newcommand{\cs}{{\cal S}}
\newcommand{\s}{\sigma}
\newcommand{\uad}{{\cal U}_{\rm ad}}
\newcommand{\iTT}{\int_0^T}
\newcommand{\ioma}{\int_\Omega}
\def\vp{\phi}
\newcommand{\dx}{\,{\rm d}x}
\newcommand{\dt}{\,{\rm d}t}
\newcommand{\oma}{\Omega}
\newcommand{\ur}{{\cal U}_R}
\newcommand{\bu}{{u}_*}
\newcommand{\bmu}{{\mu}_*}
\newcommand{\bs}{{\sigma}_*}
\newcommand{\bphi}{{\phi}_*}

\newtheorem{theorem}{\textbf{Theorem}}[section]
\newtheorem{lemma}{\textbf{Lemma}}[section]
\newtheorem{proposition}{\textbf{Proposition}}[section]
\newtheorem{corollary}{\textbf{Corollary}}[section]
\newtheorem{remark}{\textbf{Remark}}[section]
\newtheorem{definition}{\textbf{Definition}}[section]

\allowdisplaybreaks[4]

\def\bt{\begin{theorem}}
\def\et{\end{theorem}}
\def\bl{\begin{lemma}}
\def\el{\end{lemma}}
\def\br{\begin{remark}}
\def\er{\end{remark}}
\def\bp{\begin{proposition}}
\def\ep{\end{proposition}}
\def\bc{\begin{corollary}}
\def\ec{\end{corollary}}
\def\bd{\begin{definition}}
\def\ed{\end{definition}}

\begin{document}

\title{Long-time Dynamics and Optimal Control \\
of a Diffuse Interface Model for Tumor Growth}

\author{
  Cecilia Cavaterra
  \footnote{Dipartimento di Matematica,
Universit\`a degli Studi di Milano, Via Saldini 50, 20133 Milano,
Italy and Istituto di Matematica Applicata e Tecnologie Informatiche ``Enrico Magenes'', CNR,
Via Ferrata 1, 27100 Pavia, Italy. \texttt{cecilia.cavaterra@unimi.it}}
 \and
  Elisabetta Rocca
  \footnote{Dipartimento di Matematica,
  Universit\`a degli Studi di Pavia, Via Ferrata 5, 27100 Pavia, Italy and
Istituto di Matematica Applicata e Tecnologie Informatiche ``Enrico Magenes'', CNR,
Via Ferrata 1, 27100 Pavia, Italy. \texttt{elisabetta.rocca@unipv.it}}
  \and
  Hao Wu
  \footnote{School of Mathematical Sciences; Key Laboratory of Mathematics for Nonlinear Sciences (Fudan University), Ministry of Education; Shanghai
Key Laboratory for Contemporary Applied Mathematics, Fudan
University, Han Dan Road 220, 200433 Shanghai, China.
    \texttt{haowufd@fudan.edu.cn}}
}

\date{\today}

\maketitle


\begin{abstract}

We investigate the long-time dynamics and optimal control problem of a diffuse interface model that describes the growth of a tumor in presence
of a nutrient and surrounded by host tissues. The state system consists of a Cahn-Hilliard type equation for the tumor cell
fraction and a reaction-diffusion equation for the nutrient. The possible medication that serves to eliminate tumor cells is in terms of drugs
and is introduced into the system through the nutrient. In this setting, the control variable acts as an external source in the nutrient equation.
First, we consider the problem of ``long-time treatment'' under a suitable given source and prove the convergence of any global solution to a
single equilibrium as $t\to+\infty$.
Then we consider the ``finite-time treatment'' of a tumor, which corresponds to an optimal control problem.
Here we also allow the objective cost functional to depend on a free time variable, which represents the unknown treatment time to be
optimized. We prove the existence of an optimal control and obtain first order necessary optimality conditions for both the drug concentration
and the treatment time.
One of the main aim of the control problem is to realize in the best possible way a desired final distribution of the tumor cells,
which is expressed by the target function $\phi_\Omega$. By establishing the Lyapunov stability of certain equilibria of the state system
(without external source), we see that $\phi_{\Omega}$ can be taken as a stable configuration, so that the tumor will not grow again once the
finite-time treatment is completed.

\end{abstract}



\section{Introduction}
\setcounter{equation}{0}

Modeling tumor growth dynamics has recently become a major issue in applied mathematics (see, for instance, \cite{CL10,LF10}, cf. also \cite{AM,Oden13}).
Roughly speaking, the models can be divided into two broad categories: continuum models and discrete or cellular automata models
(however, see, e.g., \cite[Chap.7]{CL10} for some hybrid continuum-discrete models). Concerning the former ones and in particular within the framework of
diffuse interface models, for a young tumor, before the development of quiescent cells, the resulting systems often consist of a Cahn--Hilliard type equation
(cf. \cite{CH}) for the tumor cell fraction coupled with an advection-reaction-diffusion equation for the nutrient concentration (e.g., oxygen)
\cite{CGH15,CGMR,CGRS1,CGRS2,CGRS16,FGR15, GLNeumann,GLNS,GLR17,HKNZ15,HZO,MRS,Oden10,Oden13,RS17}.
More sophisticated models taking into account multiphase tumors or also accounting for the macroscopic cell velocities that usually satisfy a generalized
Darcy's (or Brinkman's) law have been recently studied, e.g., in \cite{BCG,CL10,CLLW09,DFRSS17,EG18,FLRS,GLSS16,GLDarcy,GLNS,LF10,MR,JiangWuZheng14}.
Besides, numerical simulations of diffuse-interface models for tumor growth have been carried out in several papers (see, for instance, \cite[Chap.8]{CL10},
\cite{Ciarletta}, and references therein). Nonetheless, a rigorous mathematical analysis of the resulting systems of partial differential equations is still
in its infancy. To the best of our knowledge, the first related papers are concerned with the so-called Cahn-Hilliard-Hele-Shaw system
(see \cite{LTZ,WW,WZ}), in which the nutrient is neglected. Besides, there are recent contributions (see \cite{CGH15} and \cite{FGR15}) devoted to
analyzing a diffuse interface model proposed in \cite{HZO} and its approximations (see also \cite{HKNZ15,WZZ} and \cite{CGRS1,CGRS2,CGRS16}).

Let $\Omega\subset \mathbb{R}^d$ ($d=2,3$) be a bounded domain with smooth boundary $\partial\Omega$ and let $\nu$ denote the outward unit
normal to $\partial\Omega$. For $T\in (0,+\infty]$, we study the following coupled system of partial differential equations
\begin{align}
& \phi_t-\Delta \mu=P(\phi)(\sigma-\mu),\qquad\qquad\ \quad \text{in}\ \Omega \times (0,T), \label{p1}\\
& \mu=-\Delta \phi+F'(\phi),\qquad\qquad\qquad\qquad  \text{in}\ \Omega \times (0,T), \label{p2}\\
& \sigma_t-\Delta \sigma = -P(\phi)(\sigma-\mu)+u,\quad\qquad \text{in}\ \Omega \times (0,T), \label{p3}
\end{align}
subject to homogeneous Neumann boundary conditions
\begin{align}
\partial_\nu \phi=\partial_\nu\mu=\partial_\nu \sigma=0,\quad \text{on}\ \partial\Omega \times (0,T),\label{bc}
\end{align}
and initial conditions
\begin{align}
\phi|_{t=0}=\phi_0(x),\quad \sigma|_{t=0}=\sigma_0(x),\quad \text{in}\ \Omega.\label{ini}
\end{align}
System \eqref{p1}--\eqref{p3} is an approximation of the model proposed in \cite{HZO}.
The macroscopic velocities of cells are set to zero for the sake of simplicity.
The state variables are reduced to the tumor cell fraction $\phi$ and the nutrient concentration $\sigma$. Typically, $\phi \simeq 1$ and $\phi \simeq -1$
represent the tumor phase and the healthy tissue phase respectively, while $\sigma \simeq 1$ and $\sigma\simeq 0$ indicate in a nutrient-rich or nutrient-poor
extracellular water phase.
The unknown $\mu$ stands for the related chemical potential and the function $F$ is typically a double-well potential with equal minima at
$\phi = \pm 1$ (cf.~\cite{Ciarletta} and references therein for different possible choices of $F$). $P$ denotes a suitable proliferation function, which is in general a nonnegative and regular function of $\phi$.
The function $u$ serves as an external source in the equation for $\s$ and can be interpreted as a medication (or a nutrient supply).

In \cite{FGR15}, the system \eqref{p1}--\eqref{ini} with $u=0$ was rigorously analyzed concerning wellposedness, regularity and long-time behavior (in terms
of the global attractor), while in the recent paper \cite{MRS} the long-time behavior of solutions (in terms of attractors) has been studied for a different system introduced in \cite{GLSS16}. Let us indeed notice that, to the best of our knowledge, these are the only two contriobutions in the literature regarding the long-term dynamics of diffuse interface models for tumor growth.
While, in \cite{CGH15, CGRS1, CGRS2} various viscous approximations of the above system have been studied analytically. Later, for a fixed final
time $T>0$, a distributed optimal control problem for system \eqref{p1}--\eqref{ini} was investigated in \cite{CGRS16}, in which the function $u$ was taken as
the control.
With a simplified cost function of standard tracking type that only involves the phase function $\phi$ and the control $u$, the authors of \cite{CGRS16}
first prove the existence of an optimal control and, moreover, they show that the control-to-state operator is Fr\'echet differentiable between appropriate
Banach spaces and derive the first-order necessary optimality conditions in terms of a variational inequality involving the adjoint state variables.

We note that the choice of reactive terms in equations \eqref{p1} and \eqref{p3} is motivated by the linear phenomenological constitutive laws for chemical
reactions \cite{HZO}. As a consequence, the system \eqref{p1}--\eqref{ini} turns out to be thermodynamically consistent. In particular, when $u=0$ the unknown
pair $(\phi, \sigma)$ is a dissipative gradient flow for the total free energy (see \cite[Section 3]{HZO}, see also \cite{HKNZ15}):
\begin{align}
\mathcal{E}(\phi, \sigma)=\int_\Omega \left[\frac12|\nabla \phi|^2 + F(\phi)\right] \dx +\frac12\int_\Omega \sigma^2 \dx.
\label{E}
\end{align}
Moreover generally, under the presence of the external source $u$, we observe that any smooth solution $(\phi, \sigma)$ to problem \eqref{p1}--\eqref{ini}
satisfies the following energy identity:
\begin{align}
\frac{\mathrm{d}}{\dt} \mathcal{E}(\phi, \sigma) + \int_\Omega \Big[|\nabla \mu|^2+|\nabla \sigma|^2+ P(\phi)(\mu-\sigma)^2\Big] \dx
=\int_\Omega u\sigma \dx,\quad \forall\, t>0,\label{BEL}
\end{align}
which motives the twofold aim of the present contribution.
\begin{description}
\item[1.] \textbf{Long-time treatment of medication}. For a suitably given external source $u$, we study the long-term dynamics of problem \eqref{p1}--\eqref{ini}.
We prove that any global weak solution will converge to a single equilibrium as $t\to +\infty$ and provide an estimate on the convergence rate. The related main
result is stated in Section~\ref{main:long} (see Theorem \ref{longtime}).

In this direction, our result indicates that after certain medication (or even without medication, i.e., $u=0$), the tumor will eventually grow to a steady state
as time evolves. However, since the potential function $F$ is nonconvex due to its double-well structure, problem \eqref{p1}--\eqref{ini} may admit infinite many
steady states so that for the moment one cannot identify which exactly the unique asymptotic limit as $t\to +\infty$ will be.

\item[2.] \textbf{Finite-time treatment of medication}. We investigate a more general distributed optimal control problem (cf. \cite{CGRS16}), where we allow the
objective cost functional to depend also on a free time variable, representing the unknown treatment time to be optimized. More precisely, denoting by
$T \in (0,+\infty)$ a fixed maximal time in which the patient is allowed to undergo a medical treatment, we consider \\
    \vspace{1mm}
    \noindent \textbf{(CP)} \ \ \textit{Minimize the cost functional}
\begin{align}
{\cal J}(\vp,\s,u, \tau)=&\ \frac{\beta_Q}2\int_0^\tau\!\!\ioma |\vp-\vp_Q|^2\dx\dt\,+\,
\frac{\beta_\Omega}2\ioma |\vp(\tau)-\vp_\oma|^2\dx \nonumber\\
&+\frac{\alpha_Q}{2}\int_0^\tau\!\!\ioma|\s-\s_Q|^2\dx\dt +\frac{\beta_S}{2}\int_\Omega(1+\phi(\tau))\dx\nonumber\\
&+\,\frac{\beta_u}2\iTT\!\!\ioma |u|^2\dx\dt+\beta_T\tau,
\label{cost:i}
\end{align}
\textit{subject to the control constraint}
\begin{equation} \label{Uad:i}
u\in\uad:=\{u\in L^\infty(Q):\,\,u_{\rm {min}}\le u\le u_{\rm max} \,\,\,\mbox{a.\,e. in }\,Q\}, \quad \tau\in (0,T),
\end{equation}
\textit{and to the state system} \eqref{p1}--\eqref{ini}, \textit{where} $Q := \Omega \times (0,T)$.

Here, $\tau \in (0,T]$ represents the treatment time,  $\phi_{Q}$ and $\sigma_Q$ represent a desired evolution for the tumor cells
and for the nutrient, respectively, while $\phi_{\Omega}$ stands for desired final distribution of tumor cells.
The first three terms of $\mathcal{J}$ are of standard tracking type, as often considered in the literature of parabolic optimal control,
and the fourth term of $\mathcal{J}$ measures the size of the tumor at the end of the treatment.
The fifth term penalizes large concentrations of the cytotoxic drugs, and the sixth term of $\mathcal{J}$ penalizes long treatment times.
As it is presented in $\mathcal{J}$, a large value of $|\phi - \phi_{Q}|^{2}$ would mean that the patient suffers from the growth
of the tumor, and a large value of $|u|^2$ would mean that the patient suffers from high toxicity of the drug. We shall prove the existence
of an optimal control and derive the first-order necessary optimality conditions in terms of a variational inequality involving the adjoint state variables.
The related main results are stated in Section~\ref{main:opt}.

The variable $\tau$ can be regarded as the necessary treatment time of one cycle, i.e., the amount of time the drug is applied to the patient before the
period of rest, or the treatment time before surgery. After the treatment, the ideal situation will be either the tumor is ready for surgery or the tumor
will be stable for all time without further medication (i.e., $u=0$). This goal can be realized by making different choices of the target function
$\phi_\Omega$ in the above optimal control problem \textbf{(CP)}.
For the former case, one can simply take $\phi_{\Omega}$ to be a configuration that is suitable for surgery. While for the later case, which is of more interest
to us, we want to choose $\phi_{\Omega}$ as a ``stable'' configuration of the system, so that the tumor does not grow again once the treatment is complete.
For this purpose, we prove that any local minimizer of the total free energy $\mathcal{E}$ is Lyapunov stable provided that $u=0$ (see Theorem \ref{stability}).
As a consequence, these local energy minimizers serve as possible candidates for the target function $\phi_\Omega$. Then after completing a successful medication,
the tumor will remain close to the chosen stable configuration for all time.
\end{description}

Let us briefly describe some ingredients in the mathematical analysis. The study of long-time behavior of problem \eqref{p1}--\eqref{ini} is nontrivial,
since the nonconvexity of the free energy $\mathcal{E}$ indicates that the set of steady states may have a rather complicated structure.
For the single Cahn-Hilliard equation this difficulty can be overcome by employing the \L ojasiewicz-Simon approach \cite{S83}, see, for instance,
\cite{RH99,GW08,LW19,W07,WZ04}. We also refer to \cite{CJ,FS,HT01,WGZ07,GGW18,JiangWuZheng14} and the references cited therein for further applications.
A key property that plays an important role in the analysis of the Cahn-Hilliard equation is the conservation of mass, i.e.,
$\int_\Omega \phi(t) \dx=\int_\Omega \phi_0 \dx$ for $t\geq 0$.
However, for our coupled system \eqref{p1}--\eqref{ini} this property no longer holds, which brings us new difficulties in analysis.
Besides, quite different from the Cahn-Hilliard-Oono system considered in \cite{M11}, in which the mass  $\int_\Omega \phi(t) \dx$ is not preserved due to
possible reactions, here in our case it is not obvious how to control the mass changing rate:
$$ \frac{\mathrm{d}}{\dt}\int_\Omega \phi \dx = \int_\Omega P(\phi)(\sigma-\mu) \dx.$$
Similar problem happens to the nutrient as well, that is
$$ \frac{\mathrm{d}}{\dt}\int_\Omega \sigma \dx = -\int_\Omega P(\phi)(\sigma-\mu) \dx +\int_\Omega u \dx.$$
Nevertheless, by the special cancellation between those reactive terms in \eqref{p1} and \eqref{p3}, we see that the total mass can be determined by the initial
data and the external source:
\[
\int_\Omega (\phi(t)+\sigma(t))\, \dx=\int_\Omega(\phi_0+\sigma_0)\, \dx+\int_0^t\int_\Omega u\, \dx\, \mathrm{d}\tau,\quad \forall\, t\geq 0.
\]
This observation allows us to derive a suitable version of the \L ojasiewicz-Simon type inequality in the sprit of \cite{Zhang,WGZ07} (see Appendix).
On the other hand, we can control the mass changing rates of $\phi$ and $\sigma$ by using the extra energy dissipation related to reactive terms in the basic
energy law \eqref{BEL}, i.e., $\int_\Omega P(\phi)(\mu-\sigma)^2\dx$.
Based on the above mentioned special structure of the system, by introducing a new version of \L ojasiewicz-Simon inequality (see Lemma \ref{LS2}),
we are able to prove that every global weak solution $(\phi,\sigma)$ of problem \eqref{p1}--\eqref{ini} will converge to a certain single equilibrium
$(\phi_\infty, \sigma_\infty)$ as $t\to+\infty$ and, moreover, we obtain a polynomial decay of the solution. Besides, a nontrivial application of the
\L ojasiewicz-Simon approach further leads to the Lyapunov stability of local minimizers of the free energy $\mathcal{E}$ (we only consider the zero
external mass case $u=0$ for the sake of simplicity). To the best of our knowledge, the only contribution in the study of long-time behavior for problem
\eqref{p1}--\eqref{ini} is given in \cite{FGR15} with $u=0$, where, however, the main focus is the existence of a global attractor (cf. also \cite{MRS} for a different model). The novelty of
our work is that we provide a first contribution in the literature on the uniqueness of asymptotic limit of global solutions as well as Lyapunov stability
of steady states for the diffuse interface models on tumor growth with reaction terms like in \eqref{p1}--\eqref{p3}.

Next, we give some further comments on the optimal problem \textbf{(CP)}. As in \cite{CGRS16}, here we aim to search for a medical strategy
such that the integral over the full space-time domain of the squared amount of nutrient or drug supplied (which is restricted by the control constraints)
does not inflict any harm on the patient (which is expressed by the presence of the fifth summand in the cost functional $\mathcal{J}$).
The non-negative coefficients $\beta_Q$, $\beta_\Omega$, $\alpha_Q$, $\beta_S$, $\beta_T$, $\beta_u$ indicate importance of conflicting targets given in the
strategy, for instance, ``avoid unnecessary harm to the patient'' versus ``quality of the approximation of $\phi_Q$, $\phi_\Omega$''.
In the cost functional $\mathcal{J}$, we could also add a point-wise term of the type $\ioma |\sigma(\tau)-\sigma_\Omega|^2\dx$, or we could replace the
term $\iTT\ioma|u|^2\dx\dt$ by a $\tau$-dependent term $\int_0^\tau \int_\Omega |u|^2\dx\dt$, but both would imply that we have to look for a control $u$
in a more regular space $H^1(0,T;L^2(\Omega))$, which is less interesting in view of practical applications (cf. \cite{GLR17} for further
discussions on this issue). Besides, it is possible to replace $\beta_{T} \tau$ by a more general function $f(\tau)$ where
$f: \mathbb{R}_{\geq 0} \to \mathbb{R}_{\geq 0}$ is continuously differentiable and increasing. In practice it would be safer for the patient (and thus
more desirable) to approximate the target functions in the $L^\infty$ sense rather than in the $L^2$ sense; however,
in view of the analytical difficulties that are inherent to the highly nonlinear state system \eqref{p1}--\eqref{ini}, this presently seems to be out of reach
(see also \cite{CGRS16}). Another interesting problem would be the one including a pointwise state constraint on the variable $\phi$ of the type $|\phi(x,\tau)-\phi_\Omega(x)|\leq \epsilon$ for a.e.~$x$, which could be reduced to an $L^2$-constraint $\|\phi(\tau)-\phi_\Omega\|_{L^2(\Omega)}\leq \epsilon'$ by using possible regularity of $\phi$, $\phi_\Omega$ (if available). This leads to a more involved adjoint system and it will be the subject of future works.

Regarding the existing literature on the aspect of optimal control, we mention the works of \cite{CFGS1, CFGS2, CGS1, CGS2, hw} for the
Cahn-Hilliard equation, \cite{RS,ZL1,ZL2} for the convective Cahn-Hilliard equation, \cite{FRS,HW1} for the Cahn-Hilliard-Navier-Stokes system,
\cite{SW18} for the Cahn-Hilliard-Darcy system, \cite{CS} for Allen-Cahn equation and \cite{CRW} for a liquid crystal model.
In the context of PDE constraint optimal control for diffuse interface tumor models, we have basically two recent works:  \cite{CGRS16} and \cite{GLR17}.
In \cite{CGRS16} the objective functional is \eqref{cost:i}, with the special (simpler) choices $\beta_{S} = \beta_{T} =\alpha_Q= 0$, and the state system is
exactly \eqref{p1}--\eqref{ini} but no dependence on $\tau$ is studied. In \cite{GLR17} a different diffuse interface model resulting as a particular case of a more general model introduced in \cite{GLSS16},
is studied. There the distributed control appears in the $\phi$ equation, which is a Cahn-Hilliard type equation with a source of mass on the right hand side,
but not depending on $\mu$.
Due to the presence of the control in the Cahn-Hilliard equation, in \cite{GLR17} only the case of a regularized objective cost functional can be analyzed for
bounded controls. Finally, we would also quote the recent paper \cite{CGMR}, where the authors study the problem of sliding mode control
for a simplified version of the model introduced in \cite{GLSS16}. With our work we aim to provide a contribution to the theory of free terminal time optimal
control in the context of diffuse interface tumor models, where the control is applied in the nutrient equation.

The rest of this paper is organized as follows.
In Section 2, we formulate the general hypotheses
and state some known results regarding the well-posedness, regularity as well as continuous dependence result of the state system \eqref{p1}--\eqref{ini}.
In Section 3, we study the long-time behavior of the system \eqref{p1}--\eqref{ini} under suitable assumption on the external source $u$,
and in Section 4, we prove Lyapunov stability of local minimizers of $\mathcal{E}$ with zero mass $u=0$.
The results concerning existence and first-order necessary optimality conditions for the optimal control problem \textbf{(CP)} are shown in Section 5.
In Appendix, we give a brief derivation of an extended \L ojasiewicz-Simon type inequality.

\section{Preliminaries}
\setcounter{equation}{0}
\subsection{Notations and assumptions}
\label{notation}
Throughout this paper, for a (real) Banach space $X$ we denote by $\|\cdot\|_X$ its norm, by $X'$ its
dual space, and by $\langle\cdot,\cdot\rangle_{X',X}$ the dual pairing between $X'$ and $X$. If $X$ is a Hilbert space,
then the inner product is denoted by $(\cdot, \cdot)_X$. Next, $L^q(\Omega)$, $1 \leq q \leq \infty$ denotes the usual Lebesgue space in $\Omega$ and
$\|\cdot\|_{L^q(\Omega)}$ denotes its norm.
Similarly, $W^{m,q}(\Omega)$, $m \in \mathbb{N}$, $1 \leq q \leq \infty$, denotes the usual Sobolev space with norm $\|\cdot \|_{W^{m,p}(\Omega)}$.
When $q=2$, we simply indicate $W^{m,2}(\Omega)$ by $H^m(\Omega)$. For simplicity, the inner product in $L^2(\Omega)$ will be indicated by $(\cdot, \cdot)$.
Let $I$ be an interval of $\mathbb{R}^+$ and $X$ a Banach space, the function space $L^p(I;X)$, $1 \leq p \leq \infty$ consists of $p$-integrable
functions with values in $X$. Moreover, $C_w(I;X)$ denotes the topological vector space of all bounded and weakly continuous functions from $I$ to $X$, while
$W^{1,p}(I,X)$ $(1\leq p\leq \infty)$ stands for the space of all functions $u$ such that $u, \frac{\mathrm{d}u}{\dt}\in L^p(I;X)$, where
$\frac{\mathrm{d}u}{\dt}$ denotes the vector valued distributional derivative of $u$. Bold characters will be used to denote vector spaces.

Given any function $u \in (H^1(\Omega))'$, we define the mean value by
$$\overline{u} = |\Omega|^{-1}\langle u, 1\rangle_{(H^1)',H^1}.$$
If $u\in L^1(\Omega)$, we simply have $\overline{u}=|\Omega|^{-1}\int_\Omega u \dx$.
We will use the notations
\begin{align*}
\dot{L}^2(\Omega)&=\{u \in L^2(\Omega)\;|\;\; \overline{u}=0\},\\
H^2_N(\Omega)&=\{u \in H^2(\Omega)\;|\;\;\partial_{\nu}u=0 \;\;\text{a.e. on}\;\;\partial\Omega\},\\
H^4_N(\Omega)&=\{\varphi\in H^4(\Omega)\;|\;\;\partial_{\nu}\varphi=\partial_{\nu} \Delta \varphi=0 \;\;\text{on}\;\;\partial\Omega\}.
\end{align*}
Then we have the dense and continuous embeddings $H^2_N \subset H^1 \subset L^2 \cong L^2 \subset (H^1)' \subset (H^2_N)'$
(we omit to indicate the set $\Omega$ for the sake of brevity), where
$\langle u, v\rangle_{(H^1)', H^1}=(u, v)$ and $\langle u, w\rangle_{(H^2_N)', H^2_N}=(u,w)$
for any $u\in L^2(\Omega)$, $v\in H^1(\Omega)$
and $w\in H^2_N(\Omega)$. On the other hand, we observe that the operator $A:=-\Delta$ with its domain $D(A)=H^2_{N}(\Omega)\cap  \dot{L}^2(\Omega)$
is a positively defined, self-adjoint operator on $D(A)$ and the spectral theorem enables us to define the powers $A^s$ of $A$, for $s\in\mathbb{R}$.
Then the space $(H^1(\Omega))'$ can be endowed with the equivalent norm
$\|u\|^2_{(H^1(\Omega))'}=\|\nabla A^{-1}(u-\overline{u})\|_{L^2(\Omega)}^2+|\overline{u}|^2$.

Throughout the paper, $C\geq 0$ will stand for a generic constant and $\mathcal{Q}(\cdot)$ for a generic positive monotone increasing function.
Special dependence will be pointed out in the text if necessary. \medskip

We make the following assumptions on the nonlinear functions $P$ and $F$.
\begin{itemize}
\item[\textbf{(P1)}] $P\in C^{2}(\mathbb{R})$ is nonnegative. There exist $\alpha_1>0$ and some $q\in [1,4]$ such that, for all $s\in \mathbb{R}$,
    \begin{align}
    |P'(s)|\leq \alpha_1(1+|s|^{q-1}).
    \end{align}
\item[\textbf{(F1)}] $F=F_0+F_1$, with $F_0, F_1\in C^5(\mathbb{R})$. There exist $\alpha_i>0$, $i=2,...,6$ and $r\in [2,6)$ such that, for all
$s\in \mathbb{R}$,
\begin{align}
&|F_1''(s)|\leq \alpha_2, \\
&\alpha_3(1+|s|^{r-2})\leq F_0''(s)\leq \alpha_4(1+|s|^{r-2}),\\
&F(s)\geq \alpha_5 |s| -\alpha_6.
\end{align}

\item[\textbf{(U1)}] For any $T>0$,  $u \in L^2(0,T; L^2(\Omega))$.
\end{itemize}
\begin{remark}
Well-posedness of problem \eqref{p1}--\eqref{ini} with $u=0$ has been obtained in \cite{FGR15} under slightly weaker conditions than \textbf{(P1)}
and \textbf{(F1)}.
On the other hand, the assumptions \textbf{(P1)}, \textbf{(F1)} and \textbf{(U1)} were indispensable for the analysis of the optimal control problem in
\cite{CGRS16}. In this paper, to avoid unnecessary technical details, we do not aim to pursue optimal assumptions on functions $F$ and $P$.
\end{remark}

\subsection{Well-posedness and continuous dependence}
\label{well}

We first state the following result on well-posedness and continuous dependence of global weak solutions for problem \eqref{p1}--\eqref{ini}
(see \cite[Theorem 1, Theorem 2]{FGR15} for the autonomous case $u=0$).
\begin{proposition} \label{weak}
Assume that \textbf{(P1)}, \textbf{(F1)} and \textbf{(U1)} are satisfied.

(1) Let $\phi_0\in H^1(\Omega)$ and $\sigma_0\in L^2(\Omega)$. Then, for every $T>0$, problem \eqref{p1}--\eqref{ini} admits a unique weak solution
on $[0,T]$ such that
\begin{align*}
&\phi\in L^\infty(0,T; H^1(\Omega))\cap L^2(0, T; H^2_N(\Omega)\cap H^3(\Omega))\cap H^1(0,T; (H^1(\Omega))'),\\
&\sigma\in L^\infty(0,T; L^2(\Omega))\cap L^2(0, T; H^1(\Omega))\cap H^1(0,T; (H^1(\Omega))'),\\
&\mu\in L^2(0,T; H^1(\Omega)),\quad  \sqrt{P(\phi)}(\mu-\sigma)\in L^2(0, T; L^2(\Omega)).
\end{align*}
In addition, the following identities are satisfied, for a.e. $t\in (0,T)$ and for any $\xi \in H^1(\Omega)$,
\begin{align}
&\langle \phi_t, \xi\rangle_{(H^1)', H^1}+(\nabla \mu, \nabla \xi)=(P(\phi)(\mu-\sigma), \xi),\label{w1}\\
&\langle\sigma_t, \xi\rangle_{(H^1)', H^1}+(\nabla \sigma, \nabla \xi)=-(P(\phi)(\mu-\sigma), \xi)+(u,\xi),\label{w2}
\end{align}
as well as $\mu=-\Delta \phi+F'(\phi)$, a.e. in $\Omega\times (0,T)$,  together with the initial conditions \eqref{ini}.
Moreover, the following energy identity holds
\begin{align}
&\mathcal{E}(\phi(t), \sigma(t))
+\int_0^t(\|\nabla \mu\|_{L^2(\Omega)}^2+\|\nabla \sigma\|_{L^2(\Omega)}^2) \mathrm{d}\tau
+\int_0^t\int_\Omega P(\phi)(\mu-\sigma)^2 \dx\mathrm{d}\tau\nonumber\\
&\quad =\mathcal{E}(\phi_0, \sigma_0) + \int_0^t\int_\Omega u\sigma \dx\mathrm{d}\tau,\quad  \forall\, t\in[0,T],\label{BEL1}
\end{align}
where $\mathcal{E}$ is given by \eqref{E}.

(2) Let $(\phi_{i0}, \sigma_{i0})\in H^1(\Omega)\times L^2(\Omega)$ $(i=1,2)$ be two initial data, $u_i$ be two external source terms and $(\phi_i, \sigma_i)$
be the corresponding weak solutions, $i=1,2$. Then, the following continuous dependence estimate holds
\begin{align}
&\|\phi_1-\phi_2\|_{L^\infty(0,T; (H^1(\Omega))')\cap L^2(0,T; H^1(\Omega))}
+\|\sigma_1-\sigma_2\|_{L^\infty(0,T; (H^1(\Omega))')\cap L^2(0,T; L^2(\Omega))}\nonumber\\
&\quad \leq C_T(\|\phi_{10}-\phi_{20}\|_{(H^1(\Omega))'}+\|\sigma_{10}-\sigma_{20}\|_{(H^1(\Omega))'}+ \|u_1-u_2\|_{L^2(0,T; L^2(\Omega))}),\nonumber
\end{align}
where $C_T>0$ is a constant depending on $\|\phi_{i0}\|_{H^1(\Omega)}$, $\|\sigma_{i0}\|_{L^2(\Omega)}$, $\|u_i\|_{L^2(0,T; L^2(\Omega))}$, $\Omega$ and $T$.
\end{proposition}
\begin{remark}
We note that Proposition \ref{weak} can be proved by means of a double approximation procedure exactly as in \cite{FGR15}.
The presence of the external source term $u$ only yields a minor modification in the proof, thus the details can be omitted here.
\end{remark}

If the initial datum $(\phi_0, \sigma_0)$ is smoother, the following result on well-posedness and continuous dependence of global strong solutions
was obtained in \cite[Theorem 2.1, Theorem 2.2]{CGRS16}.

\begin{proposition}\label{strongexe}
Assume that \textbf{(P1)}, \textbf{(F1)} and \textbf{(U1)} are satisfied.
Let $\phi_0\in H^2_N(\Omega)\cap H^3(\Omega)$ and $\sigma_0\in H^1(\Omega)$.

(1) For every $T>0$, problem \eqref{p1}--\eqref{ini} admits a unique strong solution on $[0,T]$ such that
\begin{align*}
&\phi\in L^\infty(0,T; H^2_N(\Omega)\cap H^3(\Omega))\cap L^2(0,T; H^4_N(\Omega))\cap H^1(0,T; H^1(\Omega)),\\
&\mu\in L^\infty(0,T; H^1(\Omega))\cap L^2(0,T; H^2_N(\Omega)),\\
&\sigma\in C([0,T]; H^1(\Omega))\cap L^2 (0,T; H^2_N(\Omega))\cap H^1(0, T; L^2(\Omega)).
\end{align*}

(2) There exists a constant $K_1>0$, depending on $\|\phi_0\|_{H^3(\Omega)}$, $\|\sigma_0\|_{H^1(\Omega)}$, $\|u\|_{L^2(0, T; L^2(\Omega))}$, $\Omega$ and $T$
such that the strong solution $(\phi, \mu, \sigma)$ satisfies
\begin{align}
& \|\phi\|_{L^\infty(0,T;H^3(\Omega))\cap L^2(0,T; H^4(\Omega))\cap H^1(0,T; H^1(\Omega))}
+\|\mu\|_{L^\infty(0,T; H^1(\Omega))\cap L^2(0,T; H^2(\Omega))}\nonumber\\
&\quad +\|\sigma\|_{C([0,T]; H^1(\Omega))\cap L^2 (0,T; H^2_N(\Omega))\cap H^1(0, T; L^2(\Omega))}
\leq K_1.
\label{strhi}
\end{align}

(3) For arbitrary $u_i\in L^2(0,T; L^2(\Omega))$ $(i=1,2)$, $\phi_0\in H^2_N(\Omega)\cap H^3(\Omega)$ and $\sigma_0\in H^1(\Omega)$,
let $(\phi_i, \sigma_i)$ be the corresponding strong solutions. Then there exists a constant $K_2>0$, depending on $\|u_i\|_{L^2(0, T; L^2)}$, $\Omega$, $T$,
$\|\phi_0\|_{H^3}$ and $\|\sigma_0\|_{H^1}$, such that
\begin{align*}
&\|\phi_1-\phi_2\|_{L^\infty(0,T;H^1(\Omega))\cap L^2(0,T; H^3(\Omega))\cap H^1(0,T; (H^1(\Omega))')}
+\|\mu_1-\mu_2\|_{L^2(0,T; H^1(\Omega))}\\
&\quad +\|\sigma_1-\sigma_2\|_{C([0,T]; H^1(\Omega))\cap L^2 (0,T; H^2(\Omega))\cap H^1(0, T; L^2(\Omega))}
\leq K_2 \|u_1-u_2\|_{L^2(0,T; L^2(\Omega))}.
\end{align*}
\end{proposition}

\section{Long-time Dynamics}
\label{main:long}
\setcounter{equation}{0}

In order to study the long-time behavior of problem \eqref{p1}--\eqref{ini}, we make the following additional assumptions.
\begin{itemize}
\item[\textbf{(P2)}] $P(s)>0$, for all $s\in\mathbb{R}$.
\item[\textbf{(F2)}] $F(s)$ is real analytic, for all $s\in \mathbb{R}$.
\item[\textbf{(U2)}] $u \in L^1(0,+\infty; L^2(\Omega))\cap L^2(0,+\infty; L^2(\Omega))$ and satisfies the decay condition
\begin{align}
\sup\limits_{t\geq0}(1+t)^{3+\rho}\|u(t)\|_{L^2(\Omega)}<+\infty,\quad\text{for some } \rho>0.\label{au}
\end{align}
\end{itemize}

The main result of this section reads as follows.

\begin{theorem}[Convergence to equilibrium] \label{longtime}
Assume that \textbf{(P1)}, \textbf{(P2)}, \textbf{(F1)}, \textbf{(F2)} and \textbf{(U2)} are satisfied.
For any $\phi_0\in H^1(\Omega)$ and $\sigma_0\in L^2(\Omega)$, problem \eqref{p1}--\eqref{ini} admits a unique global weak solution $(\phi, \mu, \sigma)$
such that
\begin{align}
\lim_{t\to+\infty} \left(\|\phi(t)-\phi_\infty\|_{H^2(\Omega)}+\|\sigma(t)-\sigma_\infty\|_{L^2(\Omega)}+\|\mu(t)-\mu_\infty\|_{L^2(\Omega)}\right)=0,
\end{align}
where $(\phi_\infty,  \mu_\infty, \sigma_\infty)$ satisfies
\begin{align}
\begin{cases}
&-\Delta \phi_\infty+F'(\phi_\infty)=\mu_\infty, \qquad\qquad \, \text{in}\ \Omega,\\
&\partial_\nu\phi_\infty=0,\qquad\qquad \qquad \qquad\qquad  \text{on}\ \partial\Omega,\\
& \displaystyle{\int_\Omega (\phi_\infty+\sigma_\infty) \dx = \int_\Omega (\phi_0+\sigma_0) \dx +\int_0^{+\infty}\!\!\!\int_\Omega u \dx\dt,}
\end{cases}\label{sta1}
\end{align}
with $\mu_\infty$ and $\sigma_\infty$ being two constants given by
\begin{align}
&\sigma_\infty=\mu_\infty,\label{sta2}\\
&\mu_\infty=|\Omega|^{-1}\int_\Omega F'(\phi_\infty) dx.\label{stamu1}
\end{align}
Moreover, the following estimates on convergence rate hold
\begin{align}
&\|\phi(t)-\phi_\infty\|_{H^1(\Omega)}+\|\sigma(t)-\sigma_\infty\|_{L^2(\Omega)}\leq C(1+t)^{-\min\left\{\frac{\theta}{1-2\theta},\, \frac{\rho}{2}\right\}},
\quad \forall\, t\geq 0,\label{ratehia}\\
&\|\mu(t)-\mu_\infty\|_{L^2(\Omega)}\leq  C(1+t)^{-\frac12\min\left\{\frac{\theta}{1-2\theta},\, \frac{\rho}{2}\right\}},
\quad \forall\, t\geq 0,\label{ratehib}
\end{align}
where $C>0$ is a constant depending on $\|\phi_0\|_{H^1(\Omega)}$, $\|\sigma_0\|_{L^2(\Omega)}$, $\|\phi_\infty\|_{H^1(\Omega)}$,
$\|u\|_{L^1(0,+\infty; L^2(\Omega))}$, $\|u\|_{L^2(0,+\infty; L^2(\Omega))}$ and $\Omega$; $\theta\in (0,\frac12)$ is a constant depending on $\phi_\infty$.
\end{theorem}

The proof of Theorem~\ref{longtime} consists of several steps. We first derive some uniform-in-time a priori estimates on the solution $(\phi, \mu, \sigma)$.
Then we give a characterization on the $\omega$-limit set. Finally, we prove the convergence of the
trajectories and polynomial decay by means of a proper \L ojasiewicz-Simon inequality (see Lemma \ref{LS2}).

\subsection{Uniform-in-time estimates}
\label{unif}

We proceed to derive uniform-in-time estimates for the global weak solution to problem \eqref{p1}--\eqref{ini} under the additional assumption
\textbf{(U2)} for $u$.
\medskip

\textbf{First estimate.} Adding \eqref{w1} with \eqref{w2}, testing the resultant by $\xi=1$, after integration by parts, and then integrating with
respect to time, we have
\begin{align}
\int_\Omega (\phi(t)+\sigma(t)) \dx = \int_\Omega (\phi_0+\sigma_0) \dx +\int_0^t\int_\Omega u \dx\mathrm{d}\tau,\quad \forall\, t\geq 0,\label{masscon}
\end{align}
which together with \textbf{(U2)} implies the bound of total mass:
\begin{align}
\left|\int_\Omega (\phi(t)+\sigma(t)) \dx\right|\leq \left|\int_\Omega (\phi_0+\sigma_0) \dx\right| + \|u\|_{L^1(0,+\infty; L^1(\Omega))},\quad
\forall\, t\geq 0.
\label{massbd}
\end{align}
\medskip

\textbf{Second estimate.}
By the H\"older inequality and the Cauchy--Schwarz inequality, we infer that
 \begin{align}
 \left|\int_\Omega u\sigma \dx\right|\leq \|u\|_{L^2(\Omega)}\|\sigma\|_{L^2(\Omega)}\leq \frac12 \|u\|_{L^2(\Omega)}\|\sigma\|_{L^2(\Omega)}^2
 +\frac12\|u\|_{L^2(\Omega)}.
 \end{align}
From the energy identity \eqref{BEL1}, we have
 \begin{align}
&\mathcal{E}(\phi(t), \sigma(t))+\int_0^t(\|\nabla \mu\|_{L^2(\Omega)}^2+\|\nabla \sigma\|_{L^2(\Omega)}^2) \mathrm{d}\tau
+\int_0^t\int_\Omega P(\phi)(\mu-\sigma)^2 \dx\mathrm{d}\tau\nonumber\\
&\quad \leq \mathcal{E}(\phi_0, \sigma_0) +  \frac12 \int_0^t \|u\|_{L^2}\left( \|\sigma\|_{L^2}^2 +1\right) \dt,\quad  \forall\, t\geq 0.\label{lowes1}
\end{align}
Denote $$\widehat{\mathcal{E}}(t)= \widehat{\mathcal{E}}(\phi(t), \sigma(t))+\max\{\alpha_6, 1\},$$
where the constant $\alpha_6$ is given in \textbf{(F1)}.
From the classical Bellman--Gronwall inequality (cf. \cite{Bell}), it follows that
\begin{align}
\displaystyle \widehat{\mathcal{E}}(t)&\leq \widehat{\mathcal{E}}(0) e^{\int_0^t\|u(\tau)\|_{L^2}\mathrm{d}\tau}\nonumber\\
&\leq (\mathcal{E}(\phi_0, \sigma_0)+\max\{\alpha_6, 1\}) e^{\|u\|_{L^1(0,+\infty;L^2(\Omega))}},\quad\,
\forall\, t\geq 0,
\label{lowes2}
\end{align}
which further implies
\begin{align}
& \sup_{t\geq 0} \left(\|\phi(t)\|_{H^1(\Omega)}+\|\sigma(t)\|_{L^2(\Omega)}+\|F(\phi(t))\|_{L^1(\Omega)}\right)\leq C,\label{lowes3}\\
& \int_0^{+\infty}\!\!\!\left(\|\nabla \mu\|_{L^2(\Omega)}^2+\|\nabla \sigma\|_{L^2(\Omega)}^2+\int_\Omega P(\phi)(\mu-\sigma)^2\dx\right) \dt\leq C,
\label{lowes4}
\end{align}
where $C$ is a positive constant depending on $\|\phi_0\|_{H^1(\Omega)}$, $\|\sigma_0\|_{L^2(\Omega)}$, $\|u\|_{L^1(0,+\infty;L^2(\Omega))}$, $\Omega$ and
$\alpha_6$.
\medskip

\textbf{Third estimate.} Higher-order estimates will be carried out in a formal way, which can be justified rigorously by means of a suitable
approximation procedure (cf. \cite[Section 3]{FGR15}). Testing \eqref{w1} by $\xi=\mu_t$ and \eqref{w2} by $\xi=\sigma_t$, adding the resultants together,
we obtain (see \cite[(4.3)]{FGR15}):
\begin{align}
&\frac12\frac{\mathrm{d}}{\dt}\left(\|\nabla \mu\|_{L^2(\Omega)}^2+\|\nabla \sigma\|_{L^2(\Omega)}^2+\int_\Omega P(\phi)(\mu-\sigma)^2 \dx\right)
+\|\nabla \phi_t\|_{L^2(\Omega)}^2+\|\sigma_t\|_{L^2(\Omega)}^2\nonumber\\
&\quad =\frac12\int_\Omega P'(\phi)\phi_t(\mu-\sigma)^2 \dx -\int_\Omega F''(\phi)\phi_t^2 \dx+\int_\Omega u\sigma_t \dx.\nonumber\\
&\quad :=I_1+I_2+I_3.
\label{hi1}
\end{align}
As in \cite{FGR15}, we infer from the H\"older inequality that
\begin{align}
I_1&\leq \frac12\|P'(\phi)\|_{L^2(\Omega)}\|\phi_t\|_{L^6(\Omega)}\|\mu-\sigma\|^2_{L^6(\Omega)}\nonumber\\
&\leq C\|P'(\phi)\|_{L^2(\Omega)}\|\phi_t\|_{H^1(\Omega)}\|\mu-\sigma\|_{H^1(\Omega)}^2.\nonumber
\end{align}
Next, by the Poincar\'e-Wirtinger inequality, and using \eqref{p1}, we see that
\begin{align}
\|\phi_t\|_{H^1(\Omega)}
&\leq (1+c_\Omega)\|\nabla \phi_t\|_{L^2(\Omega)}+|\Omega|^\frac12|\overline{\phi_t}|\nonumber\\
&\leq (1+c_\Omega)\|\nabla \phi_t\|_{L^2(\Omega)}+|\Omega|^{-\frac12}\left|\int_\Omega P(\phi)(\mu-\sigma) \dx\right|\nonumber\\
&\leq (1+c_\Omega)\|\nabla \phi_t\|_{L^2(\Omega)}+|\Omega|^{-\frac12}\|\sqrt{P(\phi)}\|_{L^2(\Omega)}\|\sqrt{P(\phi)}(\mu-\sigma)\|_{L^2(\Omega)}.
\label{esphith1}
\end{align}
Besides,  from the Sobolev embedding theorem, \eqref{p2}, \eqref{lowes3} and  {\bf (F1)} it follows
\begin{align}
|\overline{\mu}|
&=|\Omega|^{-1}\left|\int_\Omega (-\Delta \phi+F'(\phi))\dx\right|\nonumber\\
&=|\Omega|^{-1}\left|\int_\Omega F'(\phi)dx\right|\nonumber\\
&\leq C (1+\|\phi\|_{L^{r-1}(\Omega)}^{r-1})\nonumber\\
&\leq C(1+\|\phi\|_{H^1(\Omega)}^{r-1})\leq C.
\label{mumean}
\end{align}
Thus, by the Poincar\'e-Wirtinger inequality again, we have
\begin{align}
\|\mu\|_{H^1(\Omega)}&\leq (1+c_\Omega)\|\nabla \mu\|_{L^2(\Omega)}+|\Omega |^\frac12|\overline{\mu}|\leq C(\|\nabla \mu\|_{L^2(\Omega)}+1).
\label{esmuh1}
\end{align}
Besides, we also notice that
\begin{align}
\|\sigma\|_{H^1(\Omega)}\leq \|\nabla \sigma\|_{L^2(\Omega)}+\|\sigma\|_{L^2(\Omega)}\leq \|\nabla \sigma\|_{L^2(\Omega)}+C.
\label{sigesh1}
\end{align}
Using the above estimates,  {\bf (P1)} and \eqref{lowes3}, we deduce that
\begin{align}
I_1&\leq C\|P'(\phi)\|_{L^2(\Omega)}\left[(1+c_\Omega)\|\nabla \phi_t\|_{L^2(\Omega)}
+|\Omega|^{-\frac12}\|\sqrt{P(\phi)}\|_{L^2(\Omega)}\|\sqrt{P(\phi)}(\mu-\sigma)\|_{L^2(\Omega)}\right]\nonumber\\
&\quad \times (\|\mu\|_{H^1(\Omega)}^2+\|\sigma\|_{H^1(\Omega)}^2)\nonumber\\
&\leq C\left(\|\nabla \phi_t\|_{L^2(\Omega)}+\|\sqrt{P(\phi)}(\mu-\sigma)\|_{L^2(\Omega)}\right)(\|\nabla \mu\|_{L^2(\Omega)}^2
+\|\nabla \sigma\|_{L^2(\Omega)}^2+1)\nonumber\\
&\leq \frac14 \|\nabla \phi_t\|_{L^2(\Omega)}^2+ C\left(\|\nabla \mu\|_{L^2(\Omega)}^4+\|\nabla \sigma\|_{L^2(\Omega)}^4
+ \|\sqrt{P(\phi)}(\mu-\sigma)\|_{L^2(\Omega)}^2+1\right).\nonumber
\end{align}
Next, using the fact $\int_\Omega \Delta \mu \dx=0$ (cf. \eqref{bc}), we have
$$\|\Delta \mu\|_{(H^1(\Omega))'}=\|\nabla \mu\|_{L^2(\Omega)}.$$
It follows from equation \eqref{p1} that
\begin{align}
\|-\phi_t+P(\phi)(\mu-\sigma)\|_{(H^1(\Omega))'}=\|\nabla\mu\|_{L^2(\Omega)}.\nonumber
\end{align}
Then on account of {\bf (P1)} and \eqref{lowes3}, we obtain
\begin{align}
\|\phi_t\|_{(H^1(\Omega))'}&\leq \|-\phi_t+P(\phi)(\mu-\sigma)\|_{(H^1(\Omega))'}+\|P(\phi)(\mu-\sigma)\|_{(H^1(\Omega))'}\nonumber\\
&\leq \|\nabla \mu\|_{L^2(\Omega)}+\|P(\phi)(\mu-\sigma)\|_{L^\frac65(\Omega)}\nonumber\\
&\leq \|\nabla \mu\|_{L^2(\Omega)}+\|\sqrt{P(\phi)}\|_{L^3(\Omega)}\|\sqrt{P(\phi)}(\mu-\sigma)\|_{L^2(\Omega)}\nonumber\\
&\leq \|\nabla \mu\|_{L^2(\Omega)}+C\left(1+\|\phi\|_{L^\frac{3q}{2}(\Omega)}^\frac{q}{2}\right)\|\sqrt{P(\phi)}(\mu-\sigma)\|_{L^2(\Omega)}\nonumber\\
&\leq \|\nabla \mu\|_{L^2(\Omega)}+C\|\sqrt{P(\phi)}(\mu-\sigma)\|_{L^2(\Omega)}.
\label{phittt}
\end{align}
Collecting \eqref{esphith1} and \eqref{phittt}, we get
\begin{align}
I_2&\leq \alpha_2\|\phi_t\|_{L^2(\Omega)}^2\nonumber\\
&\leq C\|\phi_t\|_{H^1(\Omega)}\|\phi_t\|_{(H^1(\Omega))'}\nonumber\\
&\leq C(\|\nabla \phi_t\|_{L^2(\Omega)}+\|\sqrt{P(\phi)}(\mu-\sigma)\|_{L^2(\Omega)})(\|\nabla \mu\|_{L^2(\Omega)}
+C\|\sqrt{P(\phi)}(\mu-\sigma)\|_{L^2(\Omega)})\nonumber\\
&\leq \frac14\|\nabla \phi_t\|_{L^2(\Omega)}^2+C \|\nabla \mu\|_{L^2(\Omega)}^2+\|\sqrt{P(\phi)}(\mu-\sigma)\|_{L^2(\Omega)}^2.\nonumber
\end{align}
Concerning $I_3$, it holds that
\begin{align}
I_3&\leq \|u\|_{L^2(\Omega)}\|\sigma_t\|_{L^2(\Omega)}\leq \frac12\|\sigma_t\|_{L^2(\Omega)}^2+\frac12\|u\|_{L^2(\Omega)}^2.\nonumber
\end{align}
Set
\begin{align}
\mathcal{A}(t):=\|\nabla \mu(t)\|_{L^2(\Omega)}^2+\|\nabla \sigma(t)\|_{L^2(\Omega)}^2+\|\sqrt{P(\phi(t))}(\mu(t)-\sigma(t))\|_{L^2(\Omega)}^2.\label{Y}
\end{align}
The estimate \eqref{lowes4} implies that
\begin{align}
\int_0^{+\infty} \mathcal{A}(t)dt<+\infty.\label{lowes4a}
\end{align}
From \eqref{hi1} and the above estimates for reminder terms $I_1,\ I_2,\ I_3$, we deduce that
\begin{align}
\frac{\mathrm{d}}{\dt}\mathcal{A}(t)+\|\nabla \phi_t\|_{L^2(\Omega)}^2+ \|\sigma_t\|_{L^2(\Omega)}^2\leq C\mathcal{A}(t)^2+\|u\|_{L^2(\Omega)}^2+C,\label{dYt}
\end{align}
where $C$ is a positive constant depending on $\|\phi_0\|_{H^1(\Omega)}$, $\|\sigma_0\|_{L^2(\Omega)}$, $\|u\|_{L^1(0,+\infty;L^2(\Omega))}$, $\Omega$,
$\alpha_1$, $\alpha_2$, $\alpha_4$ and $\alpha_6$.
Using \eqref{lowes4a} and the assumption $u\in L^2(0,+\infty; L^2(\Omega))$,
then  from \cite[Lemma 6.2.1]{Z04} we obtain, for any $\delta\in (0,1)$,
\begin{align}
 \mathcal{A}(t+\delta)\leq C(1+\delta^{-1}), \quad \forall\, t\geq 0,\label{esY}
\end{align}
where $C$ is a positive constant depending on the initial data $\|\phi_0\|_{H^1(\Omega)}$, $\|\sigma_0\|_{L^2(\Omega)}$, $\Omega$, $\alpha_1$, $\alpha_2$,
$\alpha_4$, $\alpha_6$, $\|u\|_{L^1(0,+\infty;L^2(\Omega))}$ and $\|u\|_{L^2(0,+\infty;L^2(\Omega))}$. Besides, by \cite[Lemma 6.2.1]{Z04}, we can also
conclude that
\begin{align}
\lim_{t\to +\infty} \mathcal{A}(t)=0.\label{convA}
\end{align}

\textbf{Fourth estimate.} Combining the estimates \eqref{mumean}--\eqref{sigesh1} and \eqref{esY}, we have
\begin{align}
\|\mu(t)\|_{H^1(\Omega)}\leq C, \qquad \forall\,t\geq \delta,\label{esmuh1a}\\
\|\sigma(t)\|_{H^1(\Omega)}\leq C,\qquad \forall\,t\geq \delta.\label{sigesh1a}
\end{align}
Testing \eqref{p2} by $-\Delta \phi$,  using \textbf{(F1)}, \eqref{lowes3} and \eqref{esmuh1a}, we see that
\begin{align}
\|\Delta \phi\|_{L^2(\Omega)}^2&=-\int_\Omega \Delta\phi \mu \dx+\int_\Omega F'(\phi)\Delta \phi \dx\nonumber\\
&\leq \frac12\|\mu\|_{L^2(\Omega)}^2+\frac12\|\Delta \phi\|_{L^2(\Omega)}^2-\int_\Omega F''(\phi)|\nabla \phi|^2 \dx\nonumber\\
&\leq \frac12\|\mu\|_{L^2(\Omega)}^2+\frac12\|\Delta \phi\|_{L^2(\Omega)}^2+\alpha_2\|\nabla \phi\|_{L^2(\Omega)}^2\nonumber\\
&\leq \frac12\|\Delta \phi\|_{L^2(\Omega)}^2+C.\nonumber
\end{align}
Then, by means of elliptic estimates, we obtain
\begin{align}
\|\phi(t)\|_{H^2(\Omega)}
&\leq C(\|\Delta \phi(t)\|_{L^2(\Omega)}+\|\phi(t)\|_{L^2(\Omega)})\leq C,\quad  \forall\,t\geq \delta.\nonumber
\end{align}
The above estimate combined with \textbf{(F1)}, \eqref{lowes3}, \eqref{esmuh1a} and the Sobolev embedding theorem $H^2\hookrightarrow L^\infty$ further yields
that
\begin{align}
\|\phi(t)\|_{H^3(\Omega)}
&\leq C(\|\mu(t)\|_{H^1(\Omega)}+\|F'(\phi(t))\|_{H^1(\Omega)}+\|\phi(t)\|_{L^2(\Omega)})\nonumber\\
&\leq C+ C\left(1+\|\phi(t)\|_{L^{\infty}(\Omega)}^{r-1}\right)+C\left(1+\|\phi(t)\|_{L^\infty(\Omega)}^{r-2}\right)\|\nabla \phi(t)\|_{L^2(\Omega)}\nonumber\\
&\leq C,\quad  \forall\,t\geq \delta.\label{esphih3}
\end{align}

\textbf{Fifth estimate.} We deduce from \textbf{(U2)}, \eqref{lowes4}, \eqref{esphith1}, \eqref{esY} and the inequality \eqref{dYt} that
\begin{align}
\int_{t}^{t+1} \left(\|\phi_t(\tau)\|_{H^1(\Omega)}^2+\|\sigma_t(\tau)\|_{L^2(\Omega)}^2\right)\mathrm{d}\tau\leq C, \quad \forall\, t\geq \delta.
\label{phithi}
\end{align}
On the other hand, since
\begin{align}
\|P(\phi)(\sigma-\mu)\|_{L^2(\Omega)}
&\leq \|P(\phi)\|_{L^\infty(\Omega)}(\|\sigma\|_{L^2(\Omega)}+\|\mu\|_{L^2(\Omega)}),\nonumber
\end{align}
and
\begin{align}
\|\nabla(P(\phi)(\sigma-\mu))\|_{L^2(\Omega)}
&\leq \|P'(\phi)\|_{L^\infty(\Omega)}\|\nabla \phi\|_{L^\infty(\Omega)}\left(\|\sigma\|_{L^2(\Omega)}+\|\mu\|_{L^2(\Omega)}\right)\nonumber\\
&\quad +\|P(\phi)\|_{L^\infty(\Omega)}\left(\|\nabla \mu\|_{L^2(\Omega)}+\|\nabla \sigma\|_{L^2(\Omega)}\right),\nonumber
\end{align}
it follows from \eqref{esmuh1a}--\eqref{esphih3} that
\begin{align}
\int_{t}^{t+1} \|P(\phi(\tau))(\sigma(\tau)-\mu(\tau))\|_{H^1(\Omega)}^2 \mathrm{d}\tau\leq C,\quad \forall\, t\geq \delta.\label{Phi1}
\end{align}
Then, by comparison in the equations \eqref{p1} and \eqref{p3}, we infer from \eqref{phithi}, \eqref{Phi1}, \textbf{(U2)}  and the elliptic regularity theory that
\begin{align}
\int_{t}^{t+1} \left(\|\mu(\tau)\|_{H^3(\Omega)}^2+\|\sigma(\tau)\|_{H^2(\Omega)}^2\right)\mathrm{d}\tau\leq C, \quad \forall\, t\geq \delta.\label{L2mu3S2}
\end{align}
Applying the elliptic regularity theory to \eqref{p2} together with the above estimate and \eqref{esphih3}, we finally arrive at
\begin{align}
\int_{t}^{t+1} \|\phi(\tau)\|_{H^5(\Omega)}^2 \mathrm{d}\tau\leq C, \quad \forall\, t\geq \delta.
\end{align}

\begin{remark}\label{reg}
(1) Since $\delta\in (0,1)$ can be taken arbitrary small, the estimate \eqref{esY} actually implies that the global weak solution $(\phi, \sigma)$ becomes a
strong one once $t>0$.

(2) If we assume in addition $\phi_0\in H^2_N(\Omega)\cap H^3(\Omega)$, $\sigma_0\in H^1(\Omega)$,
then it easily follows that
\begin{align}
 \mathcal{A}(t)\leq C, \quad \forall\, t\geq 0,\label{esY1}
\end{align}
where $C$ is a positive constant depending on the initial data $\|\phi_0\|_{H^3(\Omega)}$, $\|\sigma_0\|_{H^1(\Omega)}$, $\Omega$, $\alpha_1$, $\alpha_2$,
$\alpha_4$, $\alpha_6$, $\|u\|_{L^1(0,+\infty;L^2(\Omega))}$ and $\|u\|_{L^2(0,+\infty;L^2(\Omega))}$.
In a similar fashion as above, we see that the previous higher-order estimates hold for all $t\geq 0$.

(3) It is worth mentioning that our higher-order estimates improve the estimate \eqref{strhi} that was obtained in \cite{CGRS16}, under the additional assumption
\textbf{(U2)}.
\end{remark}

\subsection{Proof of Theorem \ref{longtime}: convergence to a single equilibrium}
\label{omega}

For any initial datum $(\phi_0, \sigma_0)\in H^1(\Omega)\times L^2(\Omega)$, we define the corresponding $\omega$-limit set
\begin{align}
\omega(\phi_0, \sigma_0)=&\{(\phi_{\infty}, \sigma_\infty)\in (H^2_N(\Omega)\cap H^3(\Omega))\times H^1(\Omega)\ :\
\exists\, \{t_n\}\nearrow+\infty\ \text{such that}\nonumber\\
&\quad (\phi(t_n), \sigma(t_n))\rightarrow (\phi_{\infty}, \sigma_\infty)\ \text{in}\ H^2(\Omega)\times L^2(\Omega)\}.\nonumber
\end{align}
From the uniform-in-time estimates \eqref{lowes4}, \eqref{esmuh1a}--\eqref{esphih3}, and keeping the mass constraint \eqref{masscon}
in mind, one can easily adapt the argument in \cite[Section 6]{CGH15} to deduce the following characterization on $\omega(\phi_0, \sigma_0)$.

\bp\label{omeset}
Assume that \textbf{(P1)}, \textbf{(F1)}, \textbf{(U2)} are satisfied.
For any initial datum $(\phi_0, \sigma_0)\in H^1(\Omega)\times L^2(\Omega)$, the associated $\omega$-limit set $\omega(\phi_0, \sigma_0)$ is non-empty.
For any element $(\phi_\infty, \sigma_\infty)\in \omega(\phi_0, \sigma_0)$, $\sigma_\infty$ is a constant and $(\phi_\infty, \sigma_\infty)$ satisfies the
stationary problem \eqref{sta1}. Besides, $\mu_\infty$ is a constant given by \eqref{stamu1} and the following relation holds
\begin{align}
&P(\phi_\infty)(\sigma_\infty-\mu_\infty)=0, \quad \text{a.e. in}\ \Omega.
\label{steady}
\end{align}
\ep
In particular, the relation \eqref{steady} and the positivity assumption on $P(\cdot)$ (recall \textbf{(P2)}) immediately yield the following property.
\begin{corollary} \label{omeset2}
Under the assumptions of Proposition \ref{omeset}, if, in addition, \textbf{(P2)} is fulfilled,
then $(\phi_\infty, \mu_\infty, \sigma_\infty)$ satisfies the elliptic problem \eqref{sta1} with relations \eqref{sta2} and \eqref{stamu1}.
\end{corollary}
\medskip

Next, given any initial datum $(\phi_0, \sigma_0)\in H^1(\Omega)\times L^2(\Omega)$ and source term $u$ satisfying \textbf{(U2)}, we denote by
\begin{align}
m_\infty:=|\Omega|^{-1}\left(\int_\Omega (\phi_0+\sigma_0) \dx +\int_0^{+\infty}\!\!\!\int_\Omega u \dx\dt\right)\label{m}
\end{align}
 the total mass at infinity time (cf. \eqref{masscon}). Then we are able to derive the following  gradient inequality of \L ojasiewicz-Simon type, which plays a crucial role in the
 study of long-time behavior of problem \eqref{p1}--\eqref{ini}.
\begin{lemma}[\L ojasiewicz-Simon Inequality]\label{LS2}
Let \textbf{(F1)}, \textbf{(F2)}, \textbf{(P1)}, \textbf{(P2)} and \textbf{(U2)} be satisfied.
Suppose that $(\phi_\infty, \mu_\infty, \sigma_\infty)$ is a solution
to the elliptic problem \eqref{sta1} together with \eqref{sta2}, \eqref{stamu1}  and $m_\infty$ is a constant given by \eqref{m}. Then there exist constants
    $\theta\in (0,\frac{1}{2})$ and $\beta>0$,
   depending on  $\phi_\infty$, $m_\infty$ and $\Omega$, such that for any  $(\phi,\sigma) \in H^2_N(\Omega)\times H^1(\Omega)$ satisfying
\begin{align}
   \label{neigh1}
   &\|\phi-\phi_\infty\|_{H^1(\Omega)}< \beta,\\
   &\int_\Omega (\phi +\sigma) \dx + m_u |\Omega|=
     \int_\Omega (\phi_\infty + \sigma_\infty) \dx = m_\infty|\Omega|,\label{cons}
\end{align}
    where $m_u$ is a certain constant fulfiling $|m_u|\leq |\Omega|^{-\frac12}\|u\|_{L^1(0,+\infty;L^2(\Omega))}$, then we have
   \begin{align}
   &\left\| \mu-\overline{\mu} \right\|_{(H^1(\Omega))'}
   + C\| \nabla \sigma\|_{L^2(\Omega)}+C\| \sqrt{P(\phi)}(\mu-\sigma)\|_{L^2(\Omega)}+C|m_u|^\frac12\nonumber\\
   &\quad  \geq\
   | \mathcal{E}(\phi,\sigma) -\mathcal{E}(\phi_\infty,\sigma_\infty)|^{1-\theta}.\label{LSb}
  \end{align}
  Here, $\mu=-\Delta \phi +F'(\phi)$ and $C>0$ is a constant depending on $\Omega$, $\phi_\infty$, $m_\infty$, $\|\phi\|_{H^2(\Omega)}$, $\|\sigma\|_{H^1(\Omega)}$
  and $\|u\|_{L^1(0,+\infty;L^2(\Omega))}$.
\end{lemma}
\begin{remark} The proof of Lemma \ref{LS2} will be postponed to the Appendix.
 The constant $m_u$ in \eqref{cons} can be viewed as a control on the deviation of $|\Omega|^{-1}\int_\Omega(\phi+\sigma)\dx$ from the mean value of the total
 mass at infinity time $m_\infty$ (see \eqref{muu} below for its precise definition, when it is connected with the evolution problem \eqref{p1}--\eqref{ini}).
\end{remark}

After the previous preparations, we are in a position to complete the proof of Theorem \ref{longtime}. \medskip

\textbf{Step 1. Convergence of the free energy  $\mathcal{E}(\phi(t), \sigma(t))$.}
Let $(\phi(t), \sigma(t))$ be the global weak solution to problem \eqref{p1}--\eqref{ini} subject to the initial datum
$(\phi_0, \sigma_0)\in H^1(\Omega)\times L^2(\Omega)$ and the source term $u$ satisfying $\textbf{(U2)}$. Thanks to Remark \ref{reg},  the global weak solution
$(\phi(t), \sigma(t))$ becomes a strong one as $t\geq \delta>0$ and is uniformly bounded in $H^3(\Omega)\times H^1(\Omega)$ for all $t\geq \delta$.
Then by the energy identity \eqref{BEL}, Young's inequality and \eqref{lowes3}, we are able to obtain the following energy inequality, for a.e. $t\geq \delta$,
\begin{align}
\frac{\mathrm{d}}{\dt} \mathcal{E}(\phi, \sigma)+\|\nabla \mu\|_{L^2(\Omega)}^2+\|\nabla \sigma\|_{L^2(\Omega)}^2+\int_\Omega P(\phi)(\mu-\sigma)^2 \dx
\leq K\|u\|_{L^2(\Omega)},\label{BELL}
\end{align}
where $K>0$ is a constant depending on $\|\phi_0\|_{H^1(\Omega)}$, $\|\sigma_0\|_{L^2(\Omega)}$, $\|u\|_{L^1(0,+\infty;L^2(\Omega))}$, $\Omega$ and $\alpha_6$.
Denote
\begin{align}
y(t)=\int_t^{+\infty}\|u(\tau)\|_{L^2(\Omega)}\mathrm{d}\tau.\label{ytt}
\end{align}
By assumption \textbf{(U2)}, there exist some positive constants $C_1$, $C_2$ such that
\begin{align}
 y(t)\leq C_1(1+t)^{-2-\rho}\quad \text{and}\quad \int_t^{+\infty} y(\tau) \mathrm{d}\tau \leq C_2(1+t)^{-1-\rho},\quad \forall\, t\geq 0.\label{yyy}
\end{align}
Let us introduce the auxiliary energy functional
 \begin{align}
 \widetilde{\mathcal{E}}(t)=\mathcal{E}(\phi(t), \sigma(t))+Ky(t)+\int_t^{+\infty}y(\tau)\mathrm{d}\tau. \label{AuE}
 \end{align}
It follows from \eqref{BELL} that,  for a.e. $t\geq \delta$,
\begin{align}
\label{BELb}
\frac{\mathrm{d}}{\dt}\widetilde{\mathcal{E}}(t)+\mathcal{D}(t)^2\leq 0,
\end{align}
where
\begin{align}\nonumber
\mathcal{D}(t)^2=\|\nabla \mu(t)\|_{L^2(\Omega)}^2+\|\nabla \sigma(t)\|_{L^2(\Omega)}^2+\int_\Omega P(\phi(t))(\mu(t)-\sigma(t))^2 \dx + y(t).
\end{align}
 Hence, $\widetilde{\mathcal{E}}(t)$ is non-increasing in $t$. Since $\widetilde{\mathcal{E}}(t)$ is also bounded from below by its definition \eqref{AuE},
 we deduce that there exists a constant
 $\mathcal{E}_\infty$ such that
 \begin{align}
 \lim_{t\rightarrow+\infty}\widetilde{\mathcal{E}}(t)= \mathcal{E}_{\infty},\label{limEa}
 \end{align}
 which together with the decay property \eqref{yyy} yields
 \begin{align}
 \lim_{t\rightarrow+\infty}\mathcal{E}(\phi(t), \sigma(t))= \mathcal{E}_{\infty}.\label{limE}
 \end{align}
 Recalling the definition of $\omega(\phi_0, \sigma_0)$, it is easy to see that
  $\mathcal{E}(\phi(t), \sigma(t))$ equals to the constant $\mathcal{E}_{\infty}$ on the set $\omega(\phi_0, \sigma_0)$.
  \medskip

  \textbf{Step 2. Convergence of the trajectory $(\phi(t), \sigma(t))$.} We proceed to prove that $\omega(\phi_0, \sigma_0)$ is actually a singleton.
  Since the constant limit $\sigma_\infty$ can be uniquely determined by $\phi_\infty$ via the mass constraint in \eqref{sta1} and $\mu_\infty$ can be
  uniquely determined by $\phi_\infty$ via \eqref{stamu1}, we only need to prove the uniqueness of asymptotic limit of the phase function $\phi(t)$
  as $t\to+\infty$.

  The proof of convergence of $\phi(t)$ follows from the so-called \L ojasiewicz--Simon approach \cite{S83}
  (see e.g., \cite{CJ, FS, HT01, GWZ06, RH99} for its various aplications). To this end, we set
\begin{align}
m_u(t)=|\Omega|^{-1}\int_t^{+\infty}\!\!\! \int_\Omega u(\tau) \dx \mathrm{d}\tau,\quad \forall\, t\geq 0.\label{muu}
\end{align}
Then it follows that
\begin{align}
|m_u(t)|\leq |\Omega|^{-\frac12}\int_t^{+\infty}\|u(\tau)\|_{L^2(\Omega)}\mathrm{d}\tau =|\Omega|^{-\frac12}y(t),
\label{mutes}
\end{align}
where we recall \eqref{ytt} for the definition of $y(t)$.

For every element $(\phi_\infty, \sigma_\infty)\in \omega(\phi_0,\sigma_0)$, by Lemma \ref{LS2} (with $m_u(t)$ given by \eqref{muu} and allowing
  $t\in [0,+\infty)$ in $m_u(t)$), there exist constants $\theta_{\phi_\infty}\in(0,\frac{1}{2})$ and $\beta_{\phi_\infty}>0$ such that the inequality
  \eqref{LSb} holds for $(\phi, \sigma)\in H^2_N(\Omega)\times H^1(\Omega)$ satisfying \eqref{cons} and
\begin{align}
 \phi\in \textbf{B}_{\beta_{\phi_\infty}}(\phi_\infty):=\Big\{\phi\in H^2_N(\Omega): \|\phi-\phi_\infty\|_{H^1(\Omega)}<\beta_{\phi_\infty}\Big\}.\nonumber
\end{align}
The union of balls $\{\textbf{B}_{\beta_{\phi_\infty}}(\phi_\infty): (\phi_\infty, \sigma_\infty) \in \omega(\phi_0,\sigma_0)\}$ forms an open cover of
the set $\Phi:=\{\phi_\infty: (\phi_\infty, \sigma_\infty)\in \omega(\phi_0, \sigma_0)\}$.
Recalling \eqref{sigesh1a} and \eqref{esphih3}, $\omega(\phi_0, \sigma_0)$ is compact in $H^2(\Omega)\times L^2(\Omega)$, thus we can find a finite
sub-cover $\{\textbf{B}_{\beta_i}(\phi_\infty^i):i=1,2,...,m\}$ of $\Phi$ (in the topology of $H^1(\Omega)$), where the constants $\beta_i, \theta_i$
corresponding to $\phi_\infty^i$ in Lemma \ref{LS2} are indexed by $i$. From the definition of $\omega(\phi_0,\sigma_0)$, there exists a sufficient
large time $t_0>\delta$ such that
\begin{align}
\nonumber \phi(t)\in\mathcal{U}:=\bigcup_{i=1}^m\textbf{B}_{\beta_i}(\phi^{i}_\infty), \quad \text{for}\;\; t\geq t_0.
\end{align}
Taking $\theta=\min_{i=1}^m\{\theta_i\}\in(0,\frac{1}{2})$, using Lemma \ref{LS2}, the convergence of energy \eqref{limE} and the estimate \eqref{mutes},
we deduce that, for all $t\geq t_0$, the global weak solution $(\phi(t), \sigma(t))$ satisfies
\begin{align}
   &\left\| \mu(t)-\overline{\mu}(t) \right\|_{(H^1(\Omega))'}
   + C\| \nabla \sigma(t)\|_{L^2(\Omega)} + C\|\sqrt{P(\phi(t))}(\mu(t)-\sigma(t))\|_{L^2(\Omega)}+Cy(t)^\frac12\nonumber\\
   &\quad  \geq\
   | \mathcal{E}(\phi(t),\sigma(t)) -\mathcal{E}_\infty|^{1-\theta},\nonumber
  \end{align}
which together with Poincar\'e's inequality and \eqref{esmuh1a} yields
\begin{align}
   & C\| \nabla \mu(t) \|_{L^2(\Omega)}
   + C\| \nabla \sigma(t)\|_{L^2(\Omega)} + C\|\sqrt{P(\phi(t))}(\mu(t)-\sigma(t))\|_{L^2(\Omega)}+Cy(t)^\frac12\nonumber\\
   &\quad  \geq\
   | \mathcal{E}(\phi(t),\sigma(t)) -\mathcal{E}_\infty|^{1-\theta}. \label{LSc}
  \end{align}
It follows from \eqref{yyy}--\eqref{BELb} that
\begin{align}
 \mathcal{E}(\phi(t), \sigma(t))-\mathcal{E}_{\infty}
 &\geq \int_{t}^{\infty}\mathcal{D}(\tau)^2\mathrm{d}\tau-Ky(t)-\int_t^{+\infty}y(\tau)\mathrm{d}\tau\nonumber\\
&\geq\int_t^{\infty}\mathcal{D}(\tau)^2 \mathrm{d}\tau-C(1+t)^{-(1+\rho)}.\label{eee0}
\end{align}
On the other hand, set
\begin{align}\nonumber
\zeta=\min\left\{\theta,\frac{\rho}{2(1+\rho)}\right\}\in(0,\frac{1}{2}).
\end{align}
Then on account of \eqref{LSc} and the uniform estimates \eqref{esmuh1a}--\eqref{esphih3}, we get
\begin{align}
&|\mathcal{E}(\phi(t), \sigma(t))-\mathcal{E}_{\infty}|\nonumber\\
&\leq\left[C\| \nabla \mu(t) \|_{L^2(\Omega)}
   + C\| \nabla \sigma(t)\|_{L^2(\Omega)} + C\|\sqrt{P(\phi(t))}(\mu(t)-\sigma(t))\|_{L^2(\Omega)} +Cy(t)^\frac12\right]^{\frac{1}{1-\theta}}\nonumber\\
&\leq C \left[\| \nabla \mu(t) \|_{L^2(\Omega)}
   + \| \nabla \sigma(t)\|_{L^2(\Omega)} + \|\sqrt{P(\phi(t))}(\mu(t)-\sigma(t))\|_{L^2(\Omega)}+y(t)^\frac12\right]^{\frac{1}{1-\zeta}}\nonumber\\
&\leq C\mathcal{D}(t)^{\frac{1}{1-\zeta}},\quad \forall\, t\geq t_0.
\label{ee}
\end{align}
It is also easy to verify that
\begin{align}
\int_t^{\infty}(1+\tau)^{-2(1+\rho)(1-\zeta)}\mathrm{d}\tau\leq\int_t^{\infty}(1+\tau)^{-(2+\rho)}\mathrm{d}\tau\leq (1+t)^{-(1+\rho)},
\quad \forall\, t\geq t_0.\label{eee1}
\end{align}
Hence, denote
\begin{align}\nonumber
\mathcal{Z}(t)=\mathcal{D}(t)+(1+t)^{-(1+\rho)(1-\zeta)}.
\end{align}
We infer from \eqref{eee0}--\eqref{eee1} that
\begin{align}
\int_t^{\infty}\mathcal{Z}(\tau)^2 \mathrm{d}\tau &\leq C\mathcal{D}(t)^{\frac{1}{1-\zeta}}+C(1+t)^{-(1+\rho)}
\leq  C\mathcal{Z}(t)^{\frac{1}{1-\zeta}},\quad\forall\, t\geq t_0.\label{z}
\end{align}
Recall the following lemma (see \cite{FS, HT01})
\begin{lemma}\label{f}
Let $\zeta\in(0,\frac{1}{2})$. Assume that $Z\geq0$ be a measurable function on $(\tau,+\infty)$, $Z\in L^2(\tau, +\infty)$ and there exist $C>0$ and
$t_0\geq\tau$ such that
\begin{align}
\nonumber \int_t^{\infty} Z(s)^2\mathrm{d}s\leq CZ(t)^{\frac{1}{1-\zeta}},\quad\text{for a.e.   }t\geq t_0.
\end{align}
Then $Z\in L^1(t_0,+\infty)$.
\end{lemma}
\noindent Then from \eqref{z} and Lemma \ref{f},  we can deduce that
\begin{align}
\label{z1}\int_{t_0}^{+\infty}\mathcal{Z}(t)\dt<+\infty,
\end{align}
which together with \eqref{phittt} yields
\begin{align}
\nonumber\int_{t_0}^{+\infty}\|\phi_t\|_{(H^1(\Omega))'} \dt<+\infty.
\end{align}
As a consequence, the trajectory $\phi(t)$ converges strongly in $(H^1(\Omega))'$ as $t\rightarrow+\infty$.
 Thanks to the compactness of the trajectory in $H^2(\Omega)\times L^2(\Omega)$ for $t\geq t_0$, we can conclude from Corollary \ref{omeset2}
 that there exists a steady state $(\phi_{\infty}, \sigma_\infty)\in \omega(\phi_0, \sigma_0)$ such that
\begin{align}\nonumber
\lim_{t\rightarrow+\infty}\left(\|\phi(t)-\phi_{\infty}\|_{H^2(\Omega)}+\|\sigma(t)-\sigma_\infty\|_{L^2(\Omega)}\right)=0.
\end{align}
Namely, $\omega(\phi_0, \sigma_0)$ consists of only one point $(\phi_{\infty}, \sigma_\infty)$. Then $\mu_\infty$ is also uniquely determined as well.
\medskip

\textbf{Step 3. Convergence rate.} It follows from \eqref{LSc} and \eqref{yyy}--\eqref{limEa}  that
\begin{align}
(\widetilde{\mathcal{E}}(t)-\mathcal{E}_\infty)^{2(1-\theta)}
&\leq C\mathcal{D}(t)^{2}+C(1+t)^{-2(1-\theta)(1+\rho)}\nonumber\\
&\leq-C\frac{d}{dt}(\widetilde{\mathcal{E}}(t)-\mathcal{E}_\infty)+C(1+t)^{-2(1-\theta)(1+\rho)}.\nonumber
\end{align}
Then, by \cite[Lemma 2.8]{Ha}, we obtain the decay rate of energy:
\begin{align}
\nonumber \widetilde{\mathcal{E}}(t)-\mathcal{E}_\infty\leq C(1+t)^{-\kappa},\quad \forall\, t\geq t_0,
\end{align}
with the exponent $\kappa$ given by
\begin{align}
\nonumber
\kappa=\min\left\{\frac{1}{1-2\theta},1+\rho\right\}.
\end{align}
From \eqref{BELb} we infer, for any $t\geq t_0$,
\begin{align}
\nonumber
\int_t^{2t}\mathcal{D}(\tau)\mathrm{d}\tau
\leq t^{\frac{1}{2}}\left(\int_t^{2t}\mathcal{D}(\tau)^2 \mathrm{d}\tau\right)^{\frac{1}{2}}
\leq Ct^{\frac{1}{2}}(\widetilde{\mathcal{E}}(t)-\mathcal{E}_\infty)^{\frac{1}{2}}\leq C(1+t)^{\frac{1-\kappa}{2}}.
\end{align}
Thus,
\begin{align}\nonumber
\int_t^{+\infty}\mathcal{D}(\tau)\mathrm{d}\tau
\leq \sum\limits_{j=0}^{+\infty}\int_{2^jt}^{2^{j+1}t}\mathcal{D}(\tau)\mathrm{d}\tau\leq C\sum\limits_{j=0}^{+\infty}(2^jt)^{-\lambda}\leq C(1+t)^{-\lambda},\quad
\forall\, t\geq t_0,
\end{align}
where
\begin{align}
\lambda=\frac{\kappa-1}{2}=\min\left\{\frac{\theta}{1-2\theta},\ \frac{\rho}{2}\right\}>0.\label{lamb}
\end{align}
Therefore, it holds (cf. \eqref{phittt})
\begin{align}\nonumber
\int_t^{+\infty}\|\phi_t(\tau)\|_{(H^1(\Omega))'} \mathrm{d}\tau\leq C\int_t^{+\infty}\mathcal{D}(\tau)\mathrm{d}\tau \leq C(1+t)^{-\lambda},
\quad\forall\, t\geq t_0,
\end{align}
which yields the convergence rate of $\phi$ in $(H^1(\Omega))'$ such that
\begin{align}
\|\phi(t)-\phi_{\infty}\|_{(H^1(\Omega))'}\leq C(1+t)^{-\lambda},\quad\forall\, t\geq t_0.\label{rate1}
\end{align}

Integrating \eqref{p1} over $\Omega$, and then integrating from $t$ to $+\infty$, we deduce the convergence rate of the mean value
\begin{align}
\left|\overline{\phi(t)}-\overline{\phi_\infty}\right|
&=\left|\int_t^{+\infty} \overline{\phi_t(\tau)} \mathrm{d}\tau\right|\nonumber\\
&=\left|\int_t^{+\infty}  \overline{P(\phi(\tau))(\sigma(t)-\mu(t))} \mathrm{d}\tau\right|\nonumber\\
&\leq C\|\sqrt{P(\phi(t))}\|_{L^\infty(t,+\infty; L^2(\Omega))}
\int_t^{+\infty} \left(\int_\Omega P(\phi(\tau))(\sigma(\tau)-\mu(\tau))^2 \dx\right)^\frac12 \mathrm{d}\tau \nonumber\\
&\leq C\int_t^{+\infty}\mathcal{D}(\tau)\mathrm{d}\tau\nonumber\\
&\leq C(1+t)^{-\lambda},\quad\forall\, t\geq t_0.\label{ratemeanphi}
\end{align}
In a similar manner, we have
\begin{align}
\left|\overline{\sigma(t)}-\sigma_\infty\right|
&\leq C\int_t^{+\infty}\mathcal{D}(\tau)\mathrm{d}\tau+C\int_t^{+\infty}\|u(\tau)\|_{L^2(\Omega)}\mathrm{d}\tau\nonumber\\
&\leq C(1+t)^{-\lambda},\quad\forall\, t\geq t_0.\label{ratemeansig}
\end{align}

We now proceed to prove higher-order estimates using the energy method (see e.g., \cite{WGZ07,GWZ06}). We just need to work with $t\geq t_0$ such that the
uniform  estimates \eqref{esmuh1a}--\eqref{esphih3} are available.
Set the differences of functions
\begin{align}
\hat{\phi}(t)=\phi(t)-\phi_\infty, \quad \hat{\sigma}(t)=\sigma(t)-\sigma_\infty, \quad \hat{\mu}(t)=\mu(t)-\mu_\infty.\nonumber
\end{align}
Then the triple $(\hat{\phi}, \hat{\sigma}, \hat{\mu})$ satisfies the following system
\begin{align}
&\hat{\phi}_t-\Delta \hat{\mu}=P(\phi)(\hat{\sigma}-\hat{\mu}),\label{d1}\\
&\hat{\mu}=-\Delta \hat{\phi}+F'(\phi)-F'(\phi_\infty),\label{d2}\\
&\hat{\sigma}_t-\Delta \hat{\sigma}=-P(\phi)(\hat{\sigma}-\hat{\mu}) +u,\label{d3}\\
&\partial_\nu \hat{\phi}=\partial_\nu\hat{\mu}=\partial_\nu \hat{\sigma}=0,\label{dbc}\\
&\hat{\phi}|_{t=0}=\phi_0-\phi_\infty,\quad \hat{\sigma}|_{t=0}=\sigma_0-\sigma_\infty.\label{dini}
\end{align}
Testing \eqref{d1} by $\hat{\mu}$, \eqref{d3} by $\hat{\sigma}$, respectively, then adding the resultants together, we have
\begin{align}
&\frac{\mathrm{d}}{\dt}\left(\frac12\|\nabla \hat{\phi}\|_{L^2(\Omega)}^2+ \int_\Omega \left(F(\phi) - F(\phi_\infty)
+ F'(\phi_\infty) \phi_\infty - F'(\phi_\infty)\phi\right) \dx +\frac12\|\hat{\sigma}\|_{L^2(\Omega)}^2\right)\nonumber\\
&\qquad +\|\nabla \hat{\mu}\|_{L^2(\Omega)}^2+\|\nabla \hat{\sigma}\|_{L^2(\Omega)}^2+\int_\Omega P(\phi)(\hat{\sigma}-\hat{\mu})^2 \dx\nonumber\\
&\quad = \int_\Omega u\hat{\sigma} \dx :=I_1.\label{D1}
\end{align}
Testing \eqref{d1} by  $A^{-1}(\hat{\phi}-\overline{\hat{\phi}})$, we get
\begin{align}
&\frac12 \frac{\mathrm{d}}{\dt}\|\hat{\phi}-\overline{\hat{\phi}}\|_{(H^1(\Omega))'}^2+\|\nabla \hat{\phi}\|_{L^2(\Omega)}^2\nonumber\\
&\quad =-\int_\Omega (F'(\phi)-F'(\phi_\infty)) \hat{\phi} \dx + \overline{\hat{\phi}} \int_\Omega (F'(\phi)-F'(\phi_\infty)) \dx\nonumber\\
&\qquad + \int_\Omega P(\phi)(\hat{\sigma}-\hat{\mu})A^{-1}(\hat{\phi}-\overline{\hat{\phi}}) \dx\nonumber\\
&\quad :=I_2+I_3+I_4.\label{D2}
\end{align}
Using Poincar\'e's inequality and Young's inequality, $I_1$ can be estimated as
\begin{align}
I_1&\leq \|u\|_{L^2(\Omega)}(\|\hat{\sigma}-\overline{\hat{\sigma}}\|_{L^2(\Omega)}+\|\overline{\sigma}-\overline{\sigma_\infty}\|_{L^2(\Omega)})\nonumber\\
&\leq \frac12\|\nabla \hat{\sigma}\|_{L^2(\Omega)}^2+ C\|u\|_{L^2(\Omega)}^2+C|\overline{\sigma}-\overline{\sigma_\infty}|^2\nonumber\\
& \leq \frac12\|\nabla \hat{\sigma}\|_{L^2(\Omega)}^2+ C(1+t)^{-2\lambda}.\nonumber
\end{align}
By the Newton--Leibniz formula
 \begin{align}
  F(\phi)=F(\phi_\infty) +
F'(\phi_\infty)(\phi-\phi_\infty) +\int_0^1\!\!\int_0^1 F''(sz \phi+(1-sz)\phi_\infty)z(\phi-\phi_\infty)^2
\mathrm{d}s\mathrm{d}z,\nonumber
\end{align}
we deduce from \eqref{rate1} and \eqref{ratemeanphi} that
\begin{align}
&\left|\int_\Omega \left(F(\phi) - F(\phi_\infty)+  F'(\phi_\infty) \phi_\infty -  F'(\phi_\infty)\phi\right) \dx\right|\nonumber\\
&\quad =   \left\vert \int_\Omega \int_0^1\!\!\int_0^1
F''(sz\phi+(1-sz)\phi_\infty)z\hat{\phi}^2 \mathrm{d}s\, \mathrm{d}z \dx \right\vert\nonumber\\
&\leq  \max_{s, z\in [0,1]} \|F''(sz \phi+(1-sz)\phi_\infty)\|_{L^\infty(\Omega)}\|\hat{\phi}\|^2_{L^2(\Omega)}\nonumber\\
&\leq  C\|\hat{\phi}\|^2_{L^2(\Omega)}\nonumber\\
&\leq C\|\nabla \hat{\phi}\|_{L^2}\|\hat{\phi}-\overline{\hat{\phi}}\|_{(H^1(\Omega))'}+ C|\overline{\hat{\phi}}|^2\nonumber\\
&\leq \frac14\|\nabla \hat{\phi}\|_{L^2(\Omega)}^2+ C(1+t)^{-2\lambda},\quad \forall\,t\geq t_0.
\label{cwascr10}
\end{align}
In a similar manner, we have
\begin{align}
I_2+I_3&\leq \frac14\|\nabla \hat{\phi}\|_{L^2(\Omega)}^2+ C(1+t)^{-2\lambda},\quad \forall\, t\geq t_0. \label{cwascr10a}
\end{align}
Finally, for $I_4$, it holds
\begin{align}
I_4&\leq
\|\sqrt{P(\phi)}\|_{L^\infty(\Omega)}\|\sqrt{P(\phi)}(\hat{\sigma}-\hat{\mu})\|_{L^2(\Omega)}\|\hat{\phi}-\overline{\hat{\phi}}\|_{(H^1(\Omega))'}\nonumber\\
&\leq \frac14\int_\Omega P(\phi)(\hat{\sigma}-\hat{\mu})^2\dx+ C(1+t)^{-2\lambda},\quad \forall\,t\geq t_0.
\label{cwascr10b}
\end{align}
Denote
\begin{align}
\mathcal{Y}(t)&= \frac12\|\nabla \hat{\phi}\|_{L^2(\Omega)}^2+ \int_\Omega \left(F(\phi) - F(\phi_\infty)
+ F'(\phi_\infty) \phi_\infty - F'(\phi_\infty)\phi\right)\dx \nonumber\\
&\quad +\frac12\|\hat{\sigma}\|_{L^2(\Omega)}^2+\frac12\|\hat{\phi}-\overline{\hat{\phi}}\|_{(H^1(\Omega))'}^2.
\nonumber
\end{align}
It follows that
\begin{align}
\begin{cases}
&\mathcal{Y}(t)\geq C_3(\|\nabla \hat{\phi}(t)\|_{L^2(\Omega)}^2+ \|\hat{\sigma}(t)\|_{L^2(\Omega)}^2)-C_4(1+t)^{-2\lambda},\quad \forall\,t\geq t_0,\\
&\mathcal{Y}(t)\leq C_3'(\|\nabla \hat{\phi}(t)\|_{L^2(\Omega)}^2+ \|\hat{\sigma}(t)\|_{L^2(\Omega)}^2)+C_4'(1+t)^{-2\lambda},\quad \forall\,t\geq t_0,
\end{cases}
\label{decY1}
\end{align}
and from \eqref{D1}, \eqref{D2}, we have
\begin{align}
\frac{d}{dt}\mathcal{Y}(t)+C_5\mathcal{Y}(t)\leq C_6(1+t)^{-2\lambda}.\label{decY2}
\end{align}
From the above inequality, we can easily obtain the decay estimate (cf. e.g., \cite{WGZ07})
\begin{align}
\mathcal{Y}(t)\leq C(1+t)^{-2\lambda}, \quad \forall\, t\geq t_0.
\end{align}
The above estimate together with \eqref{ratemeanphi}, \eqref{decY1} and Poincar\'e's inequality yields that
\begin{align}
\|\hat{\phi}(t)\|_{H^1(\Omega)}+\|\hat{\sigma}(t)\|_{L^2(\Omega)}\leq C(1+t)^{-\lambda}, \quad \forall\, t\geq t_0.\label{ratehi}
\end{align}
Finally, by the higher-order estimate \eqref{esphih3} and interpolation, we have
\begin{align}
\|\hat{\mu}(t)\|_{L^2(\Omega)}&\leq \|\Delta \hat{\phi}(t)\|_{L^2(\Omega)}+\|F'(\phi(t))-F'(\phi_\infty)\|_{L^2(\Omega)}\nonumber\\
&\leq \|\nabla \Delta \hat{\phi}(t)\|^\frac12_{L^2(\Omega)}\|\nabla \hat{\phi}(t)\|_{L^2(\Omega)}^\frac12
+\|F''\|_{L^3(\Omega)}\|\hat{\phi}(t)\|_{L^6(\Omega)}\nonumber\\
&\leq C\|\hat{\phi}(t)\|^\frac12_{H^1(\Omega)},\nonumber
\end{align}
which together with \eqref{ratehi} gives the decay estimate \eqref{ratehib}.

The proof of Theorem \ref{longtime} is complete.

\section{Lyapunov Stability with Zero Mass Source}
\setcounter{equation}{0}
\label{lyap}

Within this section, we always assume $u=0$. Then from \eqref{masscon} it follows that the total mass of the system \eqref{p1}--\eqref{ini} is now conserved:
\begin{align}
\int_\Omega(\phi(t)+\sigma(t))\dx=\int_\Omega(\phi_0+\sigma_0)\dx,\quad \forall\, t\geq 0.\label{mass0}
\end{align}

\begin{definition}\label{mini}
Let $m\in \mathbb{R}$ be an arbitrary given constant. Set
\begin{align}
\mathcal{Z}_m=\Big\lbrace (\phi, \sigma)\in H^1(\Omega)\times L^2(\Omega): \int_\Omega (\phi+\sigma) \dx = |\Omega| m\Big\rbrace.\label{Zm}
\end{align}
Any $(\phi^*, \sigma^*)\in \mathcal{Z}_m$ is called a \emph{local energy minimizer}
of the total energy $\mathcal{E}(\phi, \sigma)$ defined in \eqref{E}, if there
exists a constant $\chi > 0$ such that $\mathcal{E}(\phi^*, \sigma^*)\leq \mathcal{E}(\phi, \sigma)$, for all $(\phi, \sigma)\in \mathcal{Z}_m$ satisfying
$\|(\phi-\phi^*, \sigma-\sigma^*)\|_{H^1(\Omega)\times L^2(\Omega)} < \chi$. If $\chi=+\infty$, then $(\phi^*, \sigma^*)$ is called a
\emph{global energy minimizer} of $\mathcal{E}(\phi, \sigma)$ in $\mathcal{Z}_m$.
\end{definition}
\noindent
The main result of this section reads as follows.
\begin{theorem}[Lyapunov Stability] \label{stability}
Assume that \textbf{(F1)}, \textbf{(F2)}, \textbf{(P1)}, \textbf{(P2)} are satisfied and $u=0$.
Given $m\in \mathbb{R}$, let $(\phi^*, \sigma^*)$ be a local energy minimizer of $\mathcal{E}(\phi, \sigma)$ in $\mathcal{Z}_m$.
Then, for any $\epsilon>0$, there exists a constant $\eta\in (0,1)$ such that for arbitrary
initial datum $(\phi_0, \sigma_0)\in (H^2_N(\Omega)\cap H^3(\Omega))\times H^1(\Omega)$ satisfying $\int_\Omega (\phi_0+\sigma_0) \dx =|\Omega|m$ and
$\|\phi_0-\phi^*\|_{H^1(\Omega)}+\|\sigma_0-\sigma^*\|_{L^2(\Omega)}\leq \eta$,
problem \eqref{p1}--\eqref{ini} admits a unique global strong solution $(\phi, \sigma)$ such that
\begin{align}
\|\phi(t)-\phi^*\|_{H^1(\Omega)}+\|\sigma(t)-\sigma^*\|_{L^2(\Omega)} \leq \epsilon,\quad \forall\, t\geq 0.
\label{stabii}
\end{align}
Namely, any local energy minimizer of $\mathcal{E}(\phi, \sigma)$ in $\mathcal{Z}_m$ is locally Lyapunov stable.
 \end{theorem}

We first derive some properties for the critical point
of $\mathcal{E}(\phi, \sigma)$ in $\mathcal{Z}_m$.
For any given $m\in \mathbb{R}$, we consider the following problem for $(\phi, \mu, \sigma)$
\begin{align}
\begin{cases}
&-\Delta \phi+F'(\phi)=\mu, \qquad \ \text{in}\ \Omega,\\
&\partial_\nu \phi=0,\qquad\qquad \qquad \ \ \text{on}\ \partial\Omega,\\
& \displaystyle{\int_\Omega (\phi+\sigma) \dx = |\Omega|m,}
\end{cases}\label{sta1b}
\end{align}
where $\mu$ and $\sigma$ are constants given by
\begin{align}
&\sigma=\mu=|\Omega|^{-1}\int_\Omega F'(\phi) \dx.\label{sta2b}
\end{align}
Then we have
\begin{lemma} \label{criticalb}
Let assumption (\textbf{F1}) be satisfied.

(1) If $(\phi^*, \sigma^*) \in H^2_N(\Omega)\times \mathbb{R}$ is a strong solution
to problem \eqref{sta1b}--\eqref{sta2b}, then $(\phi^*, \sigma^*)$ is a critical point
of $\mathcal{E}(\phi, \sigma)$ in
$\mathcal{Z}_m$. Conversely, if
$(\phi^*, \sigma^*)$ is a critical point of $\mathcal{E}(\phi, \sigma)$ in
$\mathcal{Z}_m$, then $\phi^*\in H^2_N(\Omega)$, $\sigma^*\in\mathbb{R}$ and they satisfy \eqref{sta1b}--\eqref{sta2b}.

(2) If $(\phi^*, \sigma^*)$ is a local energy minimizer of $\mathcal{E}(\phi, \sigma)$ in $\mathcal{Z}_m$, then $(\phi^*, \sigma^*)$ is a critical point
of $\mathcal{E}(\phi, \sigma)$.

(3) The functional $\mathcal{E}(\phi, \sigma)$ has at least one
    minimizer $(\phi^*, \sigma^*) \in \mathcal{Z}_m$ such that
\begin{equation}
    \mathcal{E}(\phi^*, \sigma^*)= \displaystyle{\inf_{(\phi, \sigma)\in \mathcal{Z}_m}}\mathcal{E}(\phi, \sigma).
\end{equation}
\end{lemma}

\begin{proof}
Consider the Lagrange function
\begin{align}
\mathcal{L}(\phi, \sigma, \mu)=\mathcal{E}(\phi, \sigma)-\mu\left(\int_\Omega  (\phi+\sigma) \dx-|\Omega| m\right),\nonumber
\end{align}
where $\mu$ is a constant Lagrange multiplier for the mass constraint.
For any $(\psi, \xi)\in H^1(\Omega)\times L^2(\Omega)$ satisfying the constraint $\int_\Omega (\psi+\xi) \dx=0$, we have
\begin{align}
&\left.\frac{\mathrm{d}}{\mathrm{d}\varepsilon}\right|_{\varepsilon=0}\mathcal{L}(\phi^*+\varepsilon \psi, \sigma^*+\varepsilon\xi)\nonumber\\
&\quad =\lim_{\varepsilon\to 0} \frac{\mathcal{L}(\phi^*+\varepsilon \psi, \sigma^*+\varepsilon\xi)-\mathcal{L}(\phi^*, \sigma^*)}{\varepsilon}\nonumber\\
&\quad =\int_\Omega (\nabla \phi^*\cdot \nabla \psi+ F'(\phi^*)\psi -\mu\psi)\dx+\int_\Omega (\sigma^*-\mu)\xi \dx.\nonumber
\end{align}
Hence, the critical point of $\mathcal{E}(\phi, \sigma)$ in $\mathcal{Z}_m$, denoted by $(\phi^*, \sigma^*)$, is a weak solution of problem \eqref{sta1b}
satisfying $\sigma^*=\mu^*=|\Omega|^{-1}\int_\Omega F'(\phi^*)\dx$. Then by the elliptic regularity theory we have $\phi^*\in H^2_N(\Omega)$.
Proof for other statements of the lemma is standard and is omitted here.
\end{proof}

\textbf{Proof of Theorem \ref{stability}.}
Let $(\phi, \mu, \sigma)$ be a strong solution to problem
\eqref{p1}--\eqref{ini}. Using a similar argument like \cite[Lemma 6.4]{GGW18}, we have the following estimates
\begin{align}
&\|\nabla \mu\|_{L^2(\Omega)} \leq  \|\phi\|_{H^3(\Omega)}+\|F''(\phi)\|_{L^\infty(\Omega)}\|\nabla \phi\|_{L^2(\Omega)},\label{eeqq1}
\end{align}
and
\begin{align}
\|\phi\|_{H^3(\Omega)}\leq \overline{C} (\|\nabla \mu\|_{L^2(\Omega)}+
 \|F''(\phi)\|_{L^\infty(\Omega)}\|\phi\|_{H^1(\Omega)}+\|\phi\|_{H^1(\Omega)}),
\label{eeqq2}
\end{align}
where the positive constant $\overline{C}$ only depends on $\Omega$.
Next, recalling the definition of $\mathcal{A}(t)$ (see \eqref{Y}), we infer from \eqref{dYt} that
\begin{align}
\frac{\mathrm{d}}{\dt}\mathcal{A}(t)+\|\nabla \phi_t\|_{L^2(\Omega)}^2+\|\sigma_t\|_{L^2(\Omega)}^2\leq C_1(\mathcal{A}(t)^2+1),\label{dYta}
\end{align}
where the constant $C_1$ only depends on $\|\phi_0\|_{H^1(\Omega)}$, $\|\sigma_0\|_{L^2(\Omega)}$, $\Omega$ and $\alpha_6$.
Besides, from \eqref{BEL1} we still have
\begin{align}
\int_0^{+\infty} \mathcal{A}(t) \dt\leq C_2,\label{lowes4b}
\end{align}
where $C_2$ depends on $\|\phi_0\|_{H^1(\Omega)}$, $\|\sigma_0\|_{L^2(\Omega)}$, $\Omega$ and $\alpha_6$.\medskip

\textbf{Step 1}. For given $m\in \mathbb{R}$, let $(\phi^*, \sigma^*)$ be a local energy minimizer of $\mathcal{E}(\phi, \sigma)$ in $\mathcal{Z}_m$.
Then by Lemma \ref{criticalb}, $(\phi^*, \sigma^*)\in H^2_N(\Omega)\times \mathbb{R}$. By the elliptic regularity theory and assumption \textbf{(F1)}, it is easy
to see that $\phi^*\in H^3(\Omega)$. We consider an initial datum $(\phi_0, \sigma_0)\in (H^2_N(\Omega)\cap H^3(\Omega))\times H^1(\Omega)$ satisfying
$\int_\Omega (\phi_0+\sigma_0) \dx=|\Omega| m$. Besides, we assume that
\begin{align}
&\|\phi_0-\phi^*\|_{H^1(\Omega)}+\|\sigma_0-\sigma^*\|_{L^2(\Omega)}\leq \eta,\label{bdini}\\
&\|\phi_0\|_{H^3(\Omega)}\leq M_1,\quad \|\sigma_0\|_{H^1(\Omega)}\leq M_2,
\end{align}
where $\eta\in (0,1]$ is a constant to be chosen later and $M_1, M_2>0$ are sufficiently large but fixed numbers.
It follows that
\begin{align}
\|\phi_0\|_{H^1(\Omega)}\leq \|\phi^*\|_{H^1(\Omega)}+1,\quad \|\sigma_0\|_{L^2(\Omega)}\leq \|\sigma^*\|_{L^2(\Omega)}+1,
\label{bdinib}
\end{align}
and, as a consequence, we have
\begin{align*}
\mathcal{E}(\phi_0, \sigma_0)\leq M_3,
\end{align*}
where the constant $M_3>0$ depends on $\|\phi^*\|_{H^1(\Omega)}$, $\|\sigma^*\|_{L^2(\Omega)}$, $\Omega$, $\alpha_2$ and $\alpha_4$.\medskip

\textbf{Step 2}. For the above initial datum $(\phi_0, \sigma_0)$, problem \eqref{p1}--\eqref{ini} admits a unique strong solution $(\phi(t), \sigma(t))$
on $[0,+\infty)$ (cf. Remark \ref{reg}).
We note that
\begin{align}
\mathcal{A}(0)\leq M_4,\nonumber
\end{align}
where the constant $M_4>0$ depends on $M_1$, $M_2$, $\Omega$, $\alpha_1$, $\alpha_2$ and $\alpha_4$, but is independent of $\eta$.
From \eqref{dYta}--\eqref{bdinib}, it follows that
\begin{align}
\mathcal{A}(t)\leq M_5,\quad \forall\, t\geq 0,
\end{align}
where the constant $M_5>0$ only depends on $M_4$, $\|\phi^*\|_{H^1(\Omega)}$, $\|\sigma^*\|_{L^2(\Omega)}$, $\Omega$, $\alpha_1$, $\alpha_2$, $\alpha_4$ and
$\alpha_6$.
Then from \eqref{eeqq2} we infer that
\begin{align}
\|\phi(t)\|_{H^3(\Omega)}\leq M_6,\quad \|\sigma(t)\|_{H^1(\Omega)}\leq M_7,\quad \forall\, t\geq 0,
\end{align}
where the constants $M_6$, $M_7$ are again independent of $\eta$.
The above estimates also imply the following estimate on energy difference
\begin{align}
&|\mathcal{E}(\phi_0,\sigma_0)-\mathcal{E}(\phi(t), \sigma(t))|\nonumber\\
&\quad \leq \frac12\|\nabla(\phi(t)+\phi_0)\|_{L^2(\Omega)}\|\nabla (\phi(t)-\phi_0)\|_{L^2(\Omega)}
+\|F'\|_{L^2(\Omega)}\|\phi(t)-\phi_0\|_{L^2(\Omega)}\nonumber\\
&\qquad +\frac12\|\sigma(t)+\sigma_0\|_{L^2(\Omega)}\|\sigma(t)-\sigma_0\|_{L^2(\Omega)}\nonumber\\
&\quad \leq M_8(\|\phi(t)-\phi_0\|_{H^1(\Omega)}+\|\sigma(t)-\sigma_0\|_{L^2(\Omega)}),\nonumber
\end{align}
where $M_8>0$ depends on $\|\phi^*\|_{H^1(\Omega)}$, $\|\sigma^*\|_{L^2(\Omega)}$, $\Omega$, $\alpha_2$, $\alpha_4$ and $\alpha_6$, but is
independent of $\eta$. \medskip

\textbf{Step 3}. For any $\epsilon>0$, we choose
\begin{align}
\varpi=\min\{1, \epsilon, \chi, \beta\}>0,
\end{align}
where $\chi$ is the constant in the definition of the local minimizer (see Definition \ref{mini}) and $\beta$ is the constant associated with
$(\phi^*, \sigma^*)$ given by the \L ojasiewicz-Simon inequality (see Lemma \ref{LS2}, taking $u=0$).
For any $$\eta\in \left(0, \frac{\varpi}{2}\right],$$
 we define
\begin{align}
T_\eta=\inf\{\,t>0, \ \ \|\phi(t)-\phi^*\|_{H^1(\Omega)}+\|\sigma(t)-\sigma^*\|_{L^2(\Omega)}\geq \varpi\}.
\end{align}
Then, by \eqref{bdini} and the fact  $(\phi(t), \sigma(t))\in C([0,+\infty); H^1(\Omega)\times L^2(\Omega))$, we have $T_\eta>0$. \medskip

\textbf{Step 4}. Our aim is to prove that there exists a constant $\eta\in (0, \frac{\varpi}{2}]$ such that $T_\eta=+\infty$. The goal can be achieved by a
contradiction argument. If this is not the case, then for all $\eta\in (0, \frac{\varpi}{2}]$, it holds $T_\eta<+\infty$.
Here and after, we shall always exclude the trivial case such that there is a $t_0\in [0, T_\eta]$ and
$\mathcal{E}(\phi(t_0),\sigma(t_0))=\mathcal{E}(\phi^*, \sigma^*)$.
In this case, by virtue of the energy identity \eqref{BEL} (with $u=0$), we have $\|\phi_t\|_{(H^1(\Omega))'}=\|\sigma_t\|_{(H^1(\Omega))'}=0$,
for all $t\geq t_0$, and the evolution simply stops.

Next, we consider the case $\mathcal{E}(\phi(t),\sigma(t))>\mathcal{E}(\phi^*, \sigma^*)$, for $t\in [0,T_\eta]$. By the definition of $\varpi$, the energy
identity \eqref{BEL} (with $u=0$) and Lemma \ref{LS2} (with $m_u=0$), we see that on $[0,T_\eta]$ it holds
\begin{align}
&-\frac{\mathrm{d}}{\dt}[\mathcal{E}(\phi(t),\sigma(t))-\mathcal{E}(\phi^*, \sigma^*)]^\theta\nonumber\\
&\quad =-\theta [\mathcal{E}(\phi(t),\sigma(t))-\mathcal{E}(\phi^*, \sigma^*)]^{\theta-1}\frac{\mathrm{d}}{\dt}\mathcal{E}(\phi(t),\sigma(t))\nonumber\\
&\quad \geq C_3\theta\frac{\| \nabla \mu\|_{L^2(\Omega)}^2 + \| \nabla \sigma\|_{L^2(\Omega)}^2
+\| \sqrt{P(\phi)}(\mu-\sigma)\|_{L^2(\Omega)}^2}{\| \nabla \mu\|_{L^2(\Omega)}
+ \| \nabla \sigma\|_{L^2}+\| \sqrt{P(\phi)}(\mu-\sigma)\|_{L^2(\Omega)}}\nonumber\\
&\quad \geq C_4\theta\left(\| \nabla \mu\|_{L^2(\Omega)} + \| \nabla \sigma\|_{L^2(\Omega)}+\| \sqrt{P(\phi)}(\mu-\sigma)\|_{L^2(\Omega)}\right)\nonumber\\
&\quad \geq C_5(\|\phi_t\|_{(H^1(\Omega))'}+\|\sigma_t\|_{(H^1(\Omega))'}),
\label{lsbb}
\end{align}
where in the last inequality we used \eqref{phittt} for $\phi_t$ and a similar estimate for $\sigma_t$.

Thus, recalling that $\mathcal{E}(\phi(t),\sigma(t))$ is nonincreasing in time, by the choice of $\varpi$ (in particular, $\varpi\leq \chi$) and $T_\eta$,
we see from \eqref{lsbb} that
\begin{align}
&\int_0^{T_\eta} \left(\|\phi_t(t)\|_{(H^1(\Omega))'}+\|\sigma_t(t)\|_{(H^1(\Omega))'}\right)\dt\nonumber\\
&\quad \leq C_5^{-1}[\mathcal{E}(\phi_0,\sigma_0)-\mathcal{E}(\phi^*, \sigma^*)]^\theta\nonumber\\
&\quad \leq C_6\left(\|\phi_0-\phi^*\|_{H^1(\Omega)}^\theta+\|\sigma_0-\sigma^*\|_{L^2(\Omega)}^\theta\right),
\label{ttt}
\end{align}
where $C_6$ depends on $C_5$, $\|\phi^*\|_{H^1(\Omega)}$, $\|\sigma^*\|_{L^2(\Omega)}$, $\theta$, $\Omega$, $\alpha_2$, $\alpha_4$, $\alpha_6$, but is
independent of $\eta$.
From \eqref{ttt}, we deduce  that
\begin{align}
&\|\phi(T_\eta)-\phi^*\|_{H^1(\Omega)}+\|\sigma(T_\eta)-\sigma^*\|_{L^2(\Omega)}\nonumber\\
&\quad \leq \|\phi_0-\phi^*\|_{H^1(\Omega)}+\|\sigma_0-\sigma^*\|_{L^2(\Omega)}+\|\phi(T_\eta)-\phi_0\|_{H^1(\Omega)}
+\|\sigma(T_\eta)-\sigma_0\|_{L^2(\Omega)}\nonumber\\
&\quad \leq \|\phi_0-\phi^*\|_{H^1(\Omega)}
+\|\sigma_0-\sigma^*\|_{L^2(\Omega)}+C_7\|\phi(T_\eta)-\phi_0\|_{H^3(\Omega)}^\frac12 \|\phi(T_\eta)-\phi_0\|_{(H^1(\Omega))'}^\frac12\nonumber\\
&\qquad +C_7\|\sigma(T_\eta)-\sigma_0\|_{H^1(\Omega)}^\frac12\|\sigma(T_\eta)-\sigma_0\|_{(H^1(\Omega))'}^\frac12\nonumber\\
&\quad \leq \|\phi_0-\phi^*\|_{H^1(\Omega)}+\|\sigma_0-\sigma^*\|_{L^2(\Omega)}\nonumber\\
&\qquad + C_7(M_1+M_2+M_6+M_7)^\frac12\left[\left(\int_0^{T_\eta} \|\phi_t(t)\|_{(H^1(\Omega))'}\dt\right)^\frac12
+\left(\int_0^{T_\eta}\|\sigma_t(t)\|_{(H^1(\Omega))'}\dt\right)^\frac12\right]\nonumber\\
&\quad \leq \|\phi_0-\phi^*\|_{H^1(\Omega)}+\|\sigma_0-\sigma^*\|_{L^2(\Omega)}
+C_8\left(\|\phi_0-\phi^*\|_{H^1(\Omega)}^\frac{\theta}{2}+\|\sigma_0-\sigma^*\|_{L^2(\Omega)}^\frac{\theta}{2}\right).\nonumber
\end{align}
Finally, choosing
$$\eta=\min\left\{\frac{\varpi}{2}, \left(\frac{\varpi}{8C_8}\right)^\frac{2}{\theta}\right\},$$
we infer that
\begin{align}
&\|\phi(T_\eta)-\phi^*\|_{H^1(\Omega)}+\|\sigma(T_\eta)-\sigma^*\|_{L^2(\Omega)}\leq \frac{\varpi}{2}+\frac{\varpi}{4}<\varpi,\nonumber
\end{align}
which leads a contradiction with the definition of $T_\eta$.

The proof of Theorem \ref{stability} is complete.

\begin{remark}
The result on long-time behavior derived in Theorem \ref{longtime} can be applied to the global strong solution obtained in Theorem \ref{stability}.
Although it is still not obvious to identify the asymptotic limit $(\phi_\infty, \sigma_\infty)$, we are able to conclude that $(\phi_\infty, \sigma_\infty)$
also satisfies $\|\phi_\infty-\phi^*\|_{H^1(\Omega)}+\|\sigma_\infty-\sigma^*\|_{L^2(\Omega)}\leq \epsilon$ thanks to \eqref{stabii}.
In particular, if $(\phi^*, \sigma^*)$ is an isolated local energy minimizer then it is indeed locally asymptotic stable.
\end{remark}

\section{The Optimal Control Problem}
\setcounter{equation}{0}
\label{main:opt}

In this section we study the optimal control problem \textbf{(CP)}. Let $T\in (0,+\infty)$ be a fixed maximal time, $Q = \Omega \times (0,T)$ and
$\Sigma = \Omega \times (0,T)$. We make for the remainder of this paper the following general assumptions on the data:
\begin{itemize}
\item[{\bf (C1)}] \,\,\,$\beta_Q,\,\beta_\oma,\,\beta_S,\,\beta_u,\,\beta_T,\alpha_Q$ are nonnegative constants but not all zero.
\item[{\bf (C2)}] \,\,\,$\vp_Q\in L^2(Q)$, $\vp_\oma\in L^2(\Omega)$, $\s_\Omega\in L^2(\Omega)$, $u_{\rm min}\in L^\infty(Q)$,
$u_{\rm max}\in L^\infty(Q)$, and $u_{\rm min}\le u_{\rm max}$, a.e. in $Q$.
\item[{\bf (C3)}] \,\,\, Let $R>0$. ${\cal U}_R $ is an open set in $L^2(Q)$ such that $\uad\subset {\cal U}_R$ and
$\|u\|_{L^2(Q)}\le R$, for all $u\in \mathcal{U}_R$.
\end{itemize}

\begin{remark}\label{umore}
Let us mention the main differences between this problem and the ones considered in \cite{CGRS16} and \cite{GLR17}. Here we generalize the problem
of \cite{CGRS16} by adding the dependence on the medication time $\tau$ in the cost functional $\mathcal{J}$. Moreover, we can consider the $\tau$-dependent terms in $\phi$
in the cost functional, which we were not able to handle in the previous paper \cite{GLR17}, mainly due to the fact that the control $u$ here is imposed in the nutrient
equation instead of in the phase equation. Due to this fact, we are able to enhance the regularity results on $\phi$ without affecting the regularity of the
control $u$ (cf. Proposition~\ref{propreg} below).
\end{remark}

\subsection{Existence}
From Proposition~\ref{weak} it follows that the \emph{control-to-state operator} $\cs$,
$$u\mapsto \cs(u):=(\vp,\mu,\s)$$
is well-defined and Lipschitz continuous as a mapping from $\ur\subset L^2(Q)$ into the following space (see \cite[Remark 2]{CGRS16})
$$(L^\infty(0,T;(H^1(\Omega))')\cap L^2(0,T;H^1(\Omega)))\times L^2(0,T;(H^1(\Omega))')\times (L^\infty(0,T;(H^1(\Omega))')\cap L^2(Q)).$$
The triplet $(\phi, \mu, \sigma)$ is the unique weak solution to problem \eqref{p1}--\eqref{ini} with data
$(\phi_0, \sigma_0,u)$ over the time interval $[0, T]$. For convenience, we use the notations $\phi = {\cal S}_{1}(u)$ and $\s={\cal S}_{3} (u)$ for the
first and third component of ${\cal S}(u)$.
Then we prove the following result that implies the existence of a solution to problem \textbf{(CP)}.

\begin{theorem}[Existence of an optimal control]\label{thm:minimizer}
Assume that \textbf{(P1)}, \textbf{(F1)}, \textbf{(U1)} and \textbf{(C1)}--\textbf{(C3)} are satisfied.
Let $\phi_0\in H^2_N(\Omega)\cap H^3(\Omega)$ and $\sigma_0\in H^1(\Omega)$. Then there exists at least one minimizer
$(\phi_{*}, \s_{*}, u_{*}, \tau_{*})$ to problem \textbf{(CP)}. Namely,
$\phi_{*} = {\cal S}_{1}(u_{*})$, $\s_{*} = {\cal S}_3(u_{*})$ satisfy
\begin{align*}
\mathcal{J}(\phi_{*}, \s_{*}, u_{*}, \tau_{*}) & = \inf_{\substack{(w, s) \; \in \; \mathcal{U}_{\mathrm{ad}} \times [0,T]  \\
 \mathrm{s.t.}\,\phi \; = \; {\cal S}_{1}(w), \, \s\;=\;{\cal S}_3(w)}} \mathcal{J}(\phi, \s, w, s).
\end{align*}
\end{theorem}

\begin{proof}
As the cost functional $\mathcal{J}$ is bounded from below, we can find a minimizing sequence $(u_{n}, \tau_{n})_{n \in \N}$ with
$u_{n} \in \mathcal{U}_{\mathrm{ad}}$, $\tau_{n} \in (0,T)$ and the corresponding weak solutions $(\phi_{n}, \mu_{n}, \sigma_{n})_{n \in \N}$
on the interval $[0,T]$ with $\phi_{n}(0) = \phi_{0}$ and $\sigma_{n}(0) = \sigma_{0}$, for all $n \in \N$, such that
\begin{align*}
\lim_{n \to +\infty} \mathcal{J}(\phi_{n}, \s_n, u_{n}, \tau_{n}) = \inf_{\substack{(w, s) \; \in \; \mathcal{U}_{\mathrm{ad}} \times [0,T]  \\
 \mathrm{s.t.}\,\phi \; = \; {\cal S}_{1}(w), \, \s\;=\;{\cal S}_3(w)}} \mathcal{J}(\phi, \s, w, s).
\end{align*}
In particular, $u_{n} \in \mathcal{U}_{\mathrm{ad}}$ implies that $u_{n} $  is bounded in $L^\infty(Q)$ for all $n \in \N$.
As $\{\tau_{n}\}_{n \in \N}$ is a bounded sequence, there exists a convergent subsequence still denoted by $\{\tau_n\}$ such that
\begin{align*}
\tau_{n} \to \tau_{*} \in [0,T]\ \ \ \text{ as } n \to +\infty.
\end{align*}
Besides, arguing exactly as in \cite[Section 4]{CGRS16}, we can find a further subsequence, which is again indexed by $n$, such that
\begin{align}\nonumber
&u_n\rightharpoonup u \quad\mbox{weakly star in }\,L^\infty(Q),\\
\nonumber
&\vp_n\rightharpoonup \vp \quad\mbox{weakly star in }\,H^1(0,T;H^1(\Omega))\cap L^\infty(0,T;H^3(\Omega)),\\
&\nonumber\Delta\vp_n\rightharpoonup\Delta\vp \quad\mbox{weakly in }\,L^2(0,T;H^2_N(\Omega)),\\
&\nonumber\mu_n\rightharpoonup\mu \quad\mbox{weakly star in }\,L^\infty(0,T;H^1(\Omega))
\cap L^2(0,T;H^2_N(\Omega)),\\
&\nonumber\s_n\rightharpoonup \s\quad\mbox{{weakly star} in }\,H^1(0,T;L^2(\Omega)){{}\cap L^\infty(0,T;H^1(\Omega)){}}\cap L^2(0,T;H^2_N(\Omega)),
\end{align}
and, moreover,
\begin{align}\nonumber
&\vp_n\to\vp \quad\hbox{strongly in } {C^0([0,T]; H^2(\Omega))}\,,
\end{align}
which implies $\vp_n\to\vp$ strongly in $C^0(\overline{Q})$. Then we see that
\begin{align}
&F'(\vp_n)\to F'(\vp) \quad\mbox{and }\,\,\,P(\vp_n)\to P(\vp), \quad\mbox{strongly in }\,
C^0(\overline{Q}).
\nonumber
\end{align}
As a consequence, we are able to pass to the limit as $n\to\infty$ in problem (\ref{p1})--(\ref{ini}) (written for $(\vp_n,\mu_n,\s_n)$),
finding that $(\vp,\mu,\s)={\cal S}(u)$. Namely, the pair $((\vp,\mu,\s),u)$ is
admissible for the control problem {\bf (CP)}.
Furthermore, by the dominating convergence theorem, for all $p \in [1,\infty)$, we have
$\chi_{[0,\tau_{n}]}(t) \to \chi_{[0,\tau_{*}]}(t)$  strongly in $L^{p}(0,T)$.
Then by the strong convergence of $\phi_{n} - \phi_{Q}$ to $\phi_{*} - \phi_{Q}$ in $L^{2}(Q)$ and the strong convergence
$\chi_{[0,\tau_{n}]}(t)$ to $\chi_{[0,\tau_{*}]}(t)$ also in $L^{2}(Q)$, we deduce that, as $n \to \infty$,
\begin{align}
\int_{0}^{\tau_{n}}\!\! \int_{\Omega} \left|\phi_{n} - \phi_{Q}\right|^{2} \dx \dt
&= \int_{0}^{T} \|\phi_{n} - \phi_{Q}\|_{L^{2}(\Omega)}^{2} \chi_{[0,\tau_{n}]}(t) \dt \nonumber\\
& \longrightarrow \int_{0}^{T} \|\phi_{*} - \phi_{Q}\|_{L^{2}(\Omega)}^{2} \chi_{[0,\tau^{*}]}(t) \dt\nonumber\\
& = \int_{0}^{\tau_{*}}\!\! \int_{\Omega} \left|\phi_{*} - \phi_{Q}\right|^{2} \dx \dt.
\label{minimizer:1}
\end{align}
The convergence of other terms in $\mathcal{J}$ can be treated in similar manner. Hence, from the weak sequential lower semicontinuity
property of ${\cal J}$,  it follows that $(u_{*},\tau_{*})$ is indeed a minimizer of the control problem {\bf (CP)}.

The proof is complete.
 \end{proof}

\subsection{Differentiability of the solution operator $\mathcal{S}$ with respect to $u$}
First, we establish
the Fr\'echet differentiability of the solution operator $\mathcal{S}$ with respect to the control $u$.
For this purpose, we investigate the linearized state equation. For arbitrary, but fixed $\bu\in\ur$, let
$(\bphi,\bmu,\bs)=\cs(\bu)$. We consider {for} {any} $h\in L^2(Q)$ the linearized system
\begin{align}
\label{ls1}
&\partial_t\xi-\Delta\eta\,=\,P'(\bphi)(\bs-\bmu)\,\xi\,+\,P(\bphi)(\rho-\eta), & &\mbox{in }\,Q,\\
\label{ls2}
&\eta\,=\,-\Delta\xi+F''(\bphi)\,\xi, & &\mbox{in }\,Q,\\
\label{ls3}
&\partial_t\rho-\Delta\rho\,=\,-P'(\bphi)(\bs-\bmu)\,\xi\,-\,P(\bphi)(\rho-\eta)\,+\,h,& &\mbox{in }\,Q,\\
\label{ls4}
&\dn\xi {{}= \dn \eta {}} =\dn\rho\,=\,0, & &\mbox{on }\,\Sigma,\\
\label{ls5}
&\xi(0)=\rho(0)=0, & &\mbox{in }\,\oma.
\end{align}

Observe that the linearized system \eqref{ls1}--\eqref{ls5} is exactly the same as the one obtained in \cite[Section 3]{CGRS16}. Hence, we can simply
quote \cite[Theorems~3.1, 3.2]{CGRS16} for the well-posedness of the linearized system \eqref{ls1}--\eqref{ls5} and the Fr\'echet differentiability of
the control-to-state operator $\mathcal{S}$ with respect to $u$. More precisely, we have
\begin{proposition}
Assume that \textbf{(P1)} and \textbf{(F1)} are satisfied, $h\in L^2(Q)$. The system \eqref{ls1}--\eqref{ls5} admits a unique solution
$(\xi,\eta,\rho)$ such that
\begin{align*}
&\xi\in H^1(0,T;(H^1(\Omega))')\cap{ L^\infty(0,T;H^1(\Omega))\cap L^2(0,T;H^2_N(\Omega)\cap H^3(\Omega))}, \\
&{\eta\in L^2(0,T;H^1(\Omega)),}\quad \rho\in H^1(0,T;L^2(\Omega))\cap C^0([0,T];H^1(\Omega))\cap{L^2(0,T;H^2_N(\Omega))}.
\end{align*}
The triple $(\xi,\eta,\rho)$ fulfills \eqref{ls2}--\eqref{ls5} almost everywhere
in the respective sets, except for \eqref{ls1} and the related boundary condition in
\eqref{ls4} that are satisfied, for almost every $t\in (0,T)$, in the following sense
\begin{align*}
&\langle\xi_t(t),v\rangle_{(H^1(\Omega))', H^1(\Omega)}\,+\,\ioma\nabla\eta(t)\cdot\nabla v\dx\\
&\quad = \ioma P'(\bphi(t))(\bs(t)-\bmu(t))\,\xi(t)\,v\dx + \ioma P(\bphi(t))(\rho(t)-\eta(t))\,v\dx, \quad\forall\,v\in H^1(\Omega).
\end{align*}
Moreover, there exists some constant $K_3>0$, which depends only on $R$ and the data of the
state system, such that
\begin{align*}
&\|\xi\|_{H^1(0,t;(H^1(\Omega))')\cap L^\infty(0,t;H^1(\Omega))\cap L^2(0,t;H^3(\Omega))}
\,+\,\|\eta\|_{L^2(0,t;H^1(\Omega))}\\
&\quad +\,\|\rho\|_{H^1(0,t;L^2(\Omega))\cap C^0([0,t];H^1(\Omega))\cap L^2(0,t;H^2(\Omega))}\, \le\,K_3\,\|h\|_{L^2(0,t;L^2(\Omega))},\,\quad\forall\,t\in [0,T].
\end{align*}
\end{proposition}
\smallskip

\begin{proposition}\label{diffu}
Assume that \textbf{(P1)}, \textbf{(F1)}, \textbf{(U1)} and \textbf{(C1)}--\textbf{(C3)} are satisfied. Let
$\phi_0\in H^2_N(\Omega)\cap H^3(\Omega)$ and $\sigma_0\in H^1(\Omega)$. Then
the control-to-state operator ${\cal S}$ is Fr\'echet differentiable in $\ur$ as
a mapping from $L^2(Q)$ into the space
\begin{align*}
{\cal Y}:=&\left(H^1(0,T;(H^2_N(\Omega))')\cap L^\infty(0,T;L^2(\Omega))\cap L^2(0,T;H^2_N(\Omega))\right)
\times L^2(Q)\\
&\nonumber\quad\times\left(H^1(0,T;L^2(\Omega))\cap L^2(0,T;H^2(\Omega))\right).
\end{align*}
Moreover, for any $\bu\in\ur$,
the Fr\'echet derivative $D{\cal S}(\bu)\in {\cal L}(L^2(Q),{\cal Y})$ is defined as
follows: for any $h\in L^2(Q)$, we have $$D{\cal S}(\bu)h=(\xi^h,\eta^h,\rho^h),$$
where $(\xi^h,\eta^h,\rho^h)$ is the unique solution to the linearized system
\eqref{ls1}--\eqref{ls5} associated with $h$.
\end{proposition}

\subsection{The first order necessary optimality conditions}

Define a reduced functional
\begin{align*}
\widetilde{\mathcal{J}}(u, \tau) := \mathcal{J}(S_{1}(u), S_3(u), u, \tau).
\end{align*}
Since the embedding of $H^1(0,T;(H^2_N(\Omega))')\cap L^2(0,T;H^2_N(\Omega))$
in $C^0([0,T];L^2(\Omega))$ is continuous, from Proposition~\ref{diffu} it follows that the
control-to-state mapping ${\cal S}$ is also Fr\'echet differentiable into
$C^0([0,T];L^2(\Omega))$ with respect to $u$.
From this we deduce that the reduced cost functional $\widetilde{\mathcal{J}}$ is Fr\'echet differentiable in $\ur$.

The first order necessary optimality conditions for the minimizer $(u_*, \tau_*)$ of Theorem
\ref{thm:minimizer} also requires the Fr\'echet derivative of $\mathcal{J}$ (or equivalently, $\widetilde{\mathcal{J}}$) with respect to $\tau$.
For this purpose, we need a further regularity result on the state problem \eqref{p1}--\eqref{ini},
under a stronger assumption on the initial data $(\phi_0, \sigma_0)$.
\begin{proposition}\label{propreg}
Assume that \textbf{(P1)}, \textbf{(F1)} and \textbf{(U1)} are satisfied. Let $\phi_0\in H^2_N(\Omega)\cap H^6(\Omega)$ and $\sigma_0\in H^2_N(\Omega)$. For every
$T>0$, problem \eqref{p1}--\eqref{ini} admits a unique strong solution on $[0,T]$ such that the following extra regularity properties hold
\begin{align}
&\phi\in H^2(0,T;L^2(\Omega)),\quad \mu\in W^{1,\infty}(0,T;L^2(\Omega)).
\end{align}
Moreover, there exists a constant $K_4>0$, depending on $\|u\|_{L^2(0, T; L^2(\Omega))}$, $\Omega$, $\|\phi_0\|_{H^6(\Omega)}$ and $\|\sigma_0\|_{H^2(\Omega)}$,
such that
\[\|\phi\|_{H^2(0,T; L^2(\Omega))}+\|\mu\|_{W^{1,\infty}(0,T;L^2(\Omega))}\leq K_4.
\label{strhibis}
\]
\end{proposition}
\begin{proof}
We perform here formal a priori estimates which can be made rigorous by means of a standard approximation scheme.
Testing the time derivative of \eqref{p1} by $\phi_t$, summing it with the time derivative of \eqref{p2} tested by $\mu_t$ and then integrating in
time over $(0,T)$, from \eqref{strhi}, H\"older's inequality and Young's inequality we infer that
\begin{align*}
&\frac12\|\phi_t(t)\|_{L^2(\Omega)}^2 +\int_0^T\|\mu_t\|_{L^2(\Omega)}^2 \dt\\
&\quad \leq \frac12\|\phi_t(0)\|_{L^2(\Omega)}^2 + \int_0^T \|P'(\phi)\|_{L^\infty(\Omega)}\|\phi_t\|_{L^2(\Omega)}^2 \left(\|\sigma\|_{L^\infty(\Omega)}
+\|\mu\|_{L^\infty(\Omega)} \right)\dt\\
&\qquad + \int_0^T \|P(\phi)\|_{L^\infty(\Omega)}(\|\sigma_t\|_{L^2(\Omega)}+\|\mu_t\|_{L^2(\Omega)})\|\phi_t\|_{L^2(\Omega)}\dt\\
&\qquad + \int_0^T \|F''(\phi)\|_{L^\infty(\Omega)}\|\phi_t\|_{L^2(\Omega)}\|\mu_t\|_{L^2(\Omega)}\dt\\
&\quad \leq C+\frac12\int_0^T\|\mu_t\|_{L^2(\Omega)}^2 \dt +C\int_0^T\|\sigma_t\|_{L^2(\Omega)}^2\dt\\
&\qquad +C\int_0^T\left(1+\|\sigma\|_{H^2(\Omega)}+\|\mu\|_{H^2(\Omega)}\right)\|\phi_t\|_{L^2(\Omega)}^2 \dt\\
&\quad \leq  C+\frac12\int_0^T\|\mu_t\|_{L^2(\Omega)}^2 \dt
+C\int_0^T\left(1+\|\sigma\|_{H^2(\Omega)}+\|\mu\|_{H^2(\Omega)}\right)\|\phi_t\|_{L^2(\Omega)}^2 \dt,
\end{align*}
where $C>0$ is a constant depending on $\|\phi_0\|_{H^4(\Omega)}$, $\|\sigma_0\|_{H^1(\Omega)}$, $\|u\|_{L^2(0,T; L^2(\Omega))}$, $\Omega$ and $T$.
Using now estimate \eqref{strhi} and Gronwall's lemma, we conclude that
\begin{equation}
\label{str1}
\|\phi_t\|_{L^\infty(0,T;L^2(\Omega))}+\|\mu_t\|_{L^2(0,T;L^2(\Omega))}\leq C,
\end{equation}
where $C>0$ is a constant depending on $\|\phi_0\|_{H^4(\Omega)}$, $\|\sigma_0\|_{H^1(\Omega)}$, $\|u\|_{L^2(0,T; L^2(\Omega))}$, $\Omega$ and $T$.

Next, testing the time derivative of \eqref{p1} by $\phi_{tt}$, summing it to the second time derivative of \eqref{p2} tested by $\mu_t$ and then
integrating over $(0,T)$, from \eqref{strhi}, H\"older's inequality and Young's inequality we infer that
\begin{align*}
&\frac12\|\mu_t(t)\|_{L^2(\Omega)}^2 +\int_0^T\|\phi_{tt}\|_{L^2(\Omega)}^2 \dt\\
&\quad \leq \frac12\|\mu_t(0)\|_{L^2(\Omega)}^2
+\int_0^T \|P'(\phi)\|_{L^\infty(\Omega)}(\|\sigma\|_{L^4(\Omega)}+\|\mu\|_{L^4(\Omega)})\|\phi_t\|_{L^4(\Omega)}\|\phi_{tt}\|_{L^2(\Omega)}\dt\\
&\qquad + \int_0^T \|P(\phi)\|_{L^\infty(\Omega)}(\|\sigma_t\|_{L^2(\Omega)}+\|\mu_t\|_{L^2(\Omega)})\|\phi_{tt}\|_{L^2(\Omega)}\dt\\
&\qquad + \int_0^T \|F'''(\phi)\|_{L^\infty(\Omega)}\|\phi_t\|_{L^4(\Omega)}^2\|\mu_t\|_{L^2(\Omega)}\dt\\
&\qquad + \int_0^T \|F''(\phi)\|_{L^\infty(\Omega)}\|\phi_{tt}\|_{L^2(\Omega)}\|\mu_t\|_{L^2(\Omega)}\dt\\
&\quad \leq C + \frac12 \int_0^T\|\phi_{tt}\|_{L^2(\Omega)}^2 \dt
       +\int_0^T(\|\sigma\|_{L^4(\Omega)}^2+\|\mu\|_{L^4(\Omega)}^2)\|\phi_t\|_{L^4(\Omega)}^2 \dt\\
&\qquad +C\int_0^T(\|\sigma_t\|_{L^2(\Omega)}^2+\|\mu_t\|_{L^2(\Omega)}^2)\dt
        + C\int_0^T\|\phi_t\|_{L^4(\Omega)}^2(\|\mu_t\|_{L^2(\Omega)}^2+1) \dt\\
&\quad \leq C+ \frac12 \int_0^T\|\phi_{tt}\|_{L^2(\Omega)}^2 \dt + C\int_0^T\|\phi_t\|_{H^1(\Omega)}^2\|\mu_t\|_{L^2(\Omega)}^2 \dt,
\end{align*}
where $C>0$ is a constant depending on $\|\phi_0\|_{H^6(\Omega)}$, $\|\sigma_0\|_{H^1(\Omega)}$, $\|u\|_{L^2(0,T; L^2(\Omega))}$, $\Omega$ and $T$.
Using estimate \eqref{strhi}, from Gronwall's lemma  we deduce
\begin{equation}\label{str2}
\|\mu_t\|_{L^\infty(0,T;L^2(\Omega))}+\|\phi_{tt}\|_{L^2(0,T;L^2(\Omega))}\leq C,
\end{equation}
where $C$ is a constant depending on $\|\phi_0\|_{H^6(\Omega)}$, $\|\sigma_0\|_{H^1(\Omega)}$, $\|u\|_{L^2(0,T; L^2(\Omega))}$, $\Omega$ and $T$.
This concludes the proof of Proposition~\ref{propreg}.
\end{proof}
\smallskip

Taking advantage of the higher-order regularity result illustrated in Proposition~\ref{propreg} and using a similar argument like for \cite[Theorem 2.6]{GLR17},
we obtain
\begin{proposition}\label{difft}
Assume that \textbf{(P1)}, \textbf{(F1)}, \textbf{(U1)} and \textbf{(C1)}--\textbf{(C3)} are satisfied. For any $\phi_0\in H^2_N(\Omega)\cap H^6(\Omega)$,
$\sigma_0\in H^1(\Omega)$ and $u \in \mathcal{U}_{\mathrm{ad}}$, we denote the corresponding state variables by ${\cal S}(u) = (\phi, \mu, \sigma)$.
Assume in addition that $\phi_Q\in H^1(0,T; L^2(\Omega))$. Then the reduced functional $\widetilde{\mathcal{J}}(u, \tau)$ is Fr\'{e}chet differentiable with
respect to $\tau$ and it holds
\begin{align*}
 D_{\tau} \widetilde{\mathcal{J}}(u, \tau)
 & = \frac{\beta_{Q}}{2} \int_\Omega|\phi(\tau) - \phi_{Q}(\tau)|^2\dx +\beta_{\Omega}\int_\Omega\left(\phi(\tau)-\phi_\Omega\right)\phi_t(\tau)\dx \\
& \quad  +\frac{\alpha_Q}{2}\int_\Omega|\sigma(\tau)-\sigma_Q|^2\dx + \frac{\beta_{S}}{2} \int_{\Omega} \phi_t(\tau) \dx +\beta_T.
\end{align*}
\end{proposition}
\smallskip

Propositions \ref{diffu} and \ref{difft} allow us to derive the following first order necessary optimality condition for problem \textbf{(CP)}.
\begin{theorem}[First order necessary optimality condition]\label{necess1}
Assume that \textbf{(P1)}, \textbf{(F1)}, \textbf{(U1)}, \textbf{(C1)}--\textbf{(C3)} are satisfied,
$\phi_0\in H^2_N(\Omega)\cap H^6(\Omega)$, $\sigma_0\in H^1(\Omega)$ and $\phi_Q\in H^1(0,T; L^2(\Omega))$.
Suppose that $(\bu, \tau_*)\in\uad\times [0,T]$ is solution to the control problem {\bf (CP)} with associated state
$(\bphi,\bmu,\bs)=\cs(\bu)$.
Then, it holds
\begin{align}
&\beta_Q\int_{0}^{\Optime}\!\! \int_{\Omega} (\phi_*-\phi_Q)\xi \dx \dt
+\beta_\Omega\int_\Omega(\phi_*(\tau_*)-\phi_\Omega)\xi(\tau_*)\dx +\alpha_Q\int_0^{\Optime}\!\!\int_\Omega (\sigma_*-\sigma_Q)\rho\dx\dt
\nonumber\\
&  +\frac{\beta_S}{2}\int_\Omega \xi(\tau_*)\dx
+ \beta_{u} \int_{0}^{T}\!\! \int_{\Omega}  u_{*} (u - u_{*}) \dx \dt\geq 0,
\quad \forall\, u \in \mathcal{U}_{\mathrm{ad}},
\label{FONC:u}
\end{align}
where $(\xi, \eta, \rho)$ is the unique solution to the linearized problem \eqref{ls1}--\eqref{ls5} with $h=u-u_*$. Besides, setting
\begin{align}
{\cal L}(\phi_*,\sigma_*,\tau_*) &=   \frac{\beta_{Q}}{2} \int_\Omega|\phi_*(\tau_*) - \phi_{Q}(\tau_*)|^2\dx
+\beta_{\Omega}\int_\Omega\left(\phi_*(\tau_*)-\phi_\Omega\right)\partial_t\phi_*(\tau_*)\dx \nonumber\\
& \quad  +\frac{\alpha_Q}{2}\int_\Omega|\sigma_*(\tau_*)-\sigma_Q|^2\dx + \frac{\beta_{S}}{2} \int_{\Omega} \partial_t\phi_*(\tau_*) \dx +\beta_T\nonumber
\end{align}
we have
\begin{equation}
{\cal L}(\phi_*,\sigma_*,\tau_*)\ \
\begin{cases}
\geq 0,\quad \text{if}\ \tau_*=0,\\
= 0,\quad \text{if}\ \tau_*\in(0,T),\\
\leq 0,\quad \text{if}\ \tau_*=T.
\end{cases}
\label{FONC:tau}
\end{equation}
\end{theorem}
\begin{proof}
Recalling that $\uad$ is a closed and convex subset of $L^2(Q)$, we can conclude \eqref{FONC:u} from standard arguments (with no need to be
repeated here). On the other hand, we have $D_\tau \widetilde{\mathcal{J}}(u_*, \tau_*)(s-\tau_*)\geq 0$ for $s\in [0,T]$.
Noting that in Theorem \ref{thm:minimizer} we cannot exclude the trivial cases where $\tau_*=0$ or $\tau_*=T$, then we arrive at the conclusion \eqref{FONC:tau}.

The proof is complete.
\end{proof}

It is possible to eliminate the variables $\xi$ and $\rho$ from the inequality \eqref{FONC:u}.
Suppose that $(\bu, \tau_*)\in\uad\times [0,T]$ is solution to the control problem {\bf (CP)} with associated state
$(\bphi,\bmu,\bs)=\cs(\bu)$. By using the formal Lagrangian method (see e.g., \cite{To}), we introduce the following {\em adjoint system} (see \cite{CGRS16} for
a simplified case with $\beta_S=\beta_T=\alpha_Q=0$):
\begin{align}\label{ad1}
&-\partial_t  p+\Delta q-F''(\bphi)\, q+P'(\bphi)(\bs-\bmu)( r- p)
=\beta_Q\,(\bphi-\vp_Q), \ \quad\mbox{in }\,Q,\\[1mm]
\label{ad2}
&q-\Delta p + P(\bphi)( p- r)=0, \qquad \qquad \qquad \qquad \qquad \qquad \qquad \qquad \qquad\mbox{in }\,Q,\\[1mm]
\label{ad3}
&-\partial_t r-\Delta r + P(\bphi)( r- p)=\alpha_Q(\sigma_*-\sigma_Q), \qquad \qquad \qquad \qquad\qquad  \quad\mbox{in }\,Q,\\[1mm]
\label{ad4}
&\dn p =\dn q =\dn r=0, \qquad\qquad \qquad \qquad  \qquad \qquad \qquad \qquad \qquad \qquad \ \ \mbox{on }\,\Sigma,\\[1mm]
\label{ad5}
& r(\Optime)=0, \quad  p(\Optime)=\beta_\oma\,(\bphi(\Optime)-\vp_\Omega)+\frac{\beta_S}{2},\qquad \qquad \qquad \qquad \qquad \, \quad\mbox{in }\,\Omega\,.
\end{align}
We call $(p,q,r)$ a solution to the adjoint system
\eqref{ad1}--\eqref{ad5} if and only if the triplet $(p,q,r)$ satisfies the following conditions:
\begin{align}
&p\in H^1(0,T;(H^2_N(\Omega))')\cap C^0([0,T];L^2(\Omega))\cap L^2(0,T;H^2_N(\Omega)),\nonumber\\
&q\in L^2(Q), \nonumber\\
&\nonumber r\in H^1(0,T;L^2(\Omega))\cap C^0([0,T];H^1(\Omega))\cap L^2(0,T;H^2_N(\Omega)).\nonumber
\end{align}
The equations \eqref{ad2}--\eqref{ad5} are satisfied almost everywhere in their respective
domains, however, since the final value $p(\Optime)$ only belongs to $L^2(\Omega)$ (see \textbf{(C2)}), the regularity of $p$ is quite low so
that \eqref{ad1} and the related boundary condition in \eqref{ad4} have to be understood in the weak variational sense, namely
\begin{align}
&\langle -\partial_t p(t),v\rangle_{(H^2_N(\Omega))', H^2_N(\Omega)}\,+\ioma q(t)\Delta v\dx\,-\ioma F''(\bphi(t))\,q(t)\,v\dx
\nonumber\\[1mm]
&\quad +\ioma P'(\bphi(t))(\bs(t)-\bmu(t))\,(r(t)-p(t))\,v\dx \,=\,\ioma\beta_Q\,(\bphi(t)-\vp_Q(t))v\dx, \nonumber 
\end{align}
for all $v\in H^2_N(\Omega)$ and almost every $t\in (0,T)$.

Let us notice that the only difference between the adjoint system \eqref{ad1}--\eqref{ad5} and the one of \cite{CGRS16}
consists in the presence of a non-zero right hand side in equation \eqref{ad3} and a constant $\beta_S$ in \eqref{ad5}.
Since this extra right-hand side is indeed an $L^2(Q)$-function, we can repeat exactly the same argument used in
\cite[Theorem 4.2]{CGRS16} to obtain the following existence and uniqueness result:
\begin{proposition}
 Assume that \textbf{(P1)}, \textbf{(F1)}, \textbf{(U1)}, \textbf{(C1)}--\textbf{(C3)} are satisfied, $\phi_0\in H^2_N(\Omega)\cap H^3(\Omega)$,
 and $\sigma_0\in H^1(\Omega)$. Then the adjoint system \eqref{ad1}--\eqref{ad5} has a unique weak solution $(p,q,r)$ on $[0,T]$ in the sense formulated above.
\end{proposition}

Finally, we are able to rewrite the first order necessary optimality condition \eqref{FONC:u} for the minimizer $(u_{*}, \Optime)$ using the adjoint states:
\begin{corollary}[First order necessary optimality condition via adjoint states]\label{thm:FONC}
Assume that the assumptions of Theorem \ref{necess1} are satisfied. Let $(u_{*}, \Optime) \in \mathcal{U}_{\mathrm{ad}} \times [0,T]$ denote a minimizer to
the optimal control problem {\bf (CP)} with corresponding state variables $(\phi_{*}, \mu_{*}, \sigma_{*})={\cal S}(u_{*})$ and associated adjoint variables
$(p,q,r)$.  Then, the variational inequality \eqref{FONC:u} can be written as
\begin{equation}
\begin{aligned}
\beta_{u} \int_{0}^{T}\!\! \int_{\Omega} u_{*} (u - u_{*}) \dx \dt + \int_{0}^{\Optime}\!\! \int_{\Omega}  r (u - u_{*}) \dx \dt \geq 0, \quad
\forall\, u \in \mathcal{U}_{\mathrm{ad}}. \label{FONC:u1}
\end{aligned}
\end{equation}
\end{corollary}

\begin{remark}
For the proof we can refer to
\cite[Theorem 4.3]{CGRS16} with slight modifications due to the presence of the term $\alpha_Q(\sigma_*-\sigma_Q)$ in equation \eqref{ad3}.
Besides, if we extend $r$ by zero to $(\Optime, T]$, then we can express \eqref{FONC:u1} as
\begin{align*}
\int_{0}^{T}\!\! \int_{\Omega} (\beta_{u} u_{*} +r)(u - u_{*}) \dx \dt \geq 0, \quad \forall\, u \in \mathcal{U}_{\mathrm{ad}},
\end{align*}
which allows the interpretation that the optimal control $u_{*}$ is indeed the $L^{2}(Q)$-projection of $-\beta_{u}^{-1} r$ onto the set
$\mathcal{U}_{\mathrm{ad}}$ (provided that $\beta_u>0$).
\end{remark}

\section{Appendix}
\setcounter{equation}{0}

Let $m\in \mathbb{R}$ be a given constant. We consider the following nonlocal elliptic boundary value problem
\begin{align}
\begin{cases}
&\displaystyle{-\Delta \phi+F'(\phi)=m-|\Omega|^{-1}\int_\Omega \phi \dx, \quad \text{in}\ \Omega,}\\
&\partial_\nu\phi=0,\qquad\qquad \qquad \qquad\qquad\qquad \ \,\text{on}\ \partial\Omega.\\
\end{cases}\label{sta3}
\end{align}
Problem \eqref{sta3} can be associated with the following functional
\begin{align}
\Upsilon(\phi) = \int_{\Omega} \left(\frac{1}{2} |\nabla \phi|^2 +
F(\phi)\right)\dx+ \frac{1}{2}|\Omega|\left(m-|\Omega|^{-1}\int_\Omega \phi \dx\right)^2,\quad \forall\, \phi\in H^1(\Omega).
\label{E1}
\end{align}
The following result has been obtained in \cite[Lemma 3.1, Lemma 3.2]{WGZ07}.
\begin{lemma} \label{critical}
Let assumption (\textbf{F1}) be satisfied.

(1) Suppose that $\psi \in H^2_N(\Omega)$ is a (strong) solution
to problem \eqref{sta3}. Then $\psi$ is a critical point
of the functional $\Upsilon(\phi)$ in $H^1(\Omega)$. Conversely, if
$\psi$ is a critical point of the functional $\Upsilon(\phi)$ in
$H^1(\Omega)$, then $\psi\in H^2_N(\Omega)$ and it is a
strong solution to problem \eqref{sta3}.

(2) The functional $\Upsilon(\phi)$ has at least one
    minimizer $\psi\in H^1(\Omega)$ such that
\begin{equation}
    \Upsilon(\psi)= \displaystyle{\inf_{\phi\in H^1(\Omega)}}\Upsilon(\phi).
\end{equation}
\end{lemma}
\begin{remark}
By the elliptic regularity theory and a bootstrap argument, the minimizer $\psi$ is indeed a classical solution such that $\psi\in C^\infty(\overline{\Omega})$,
provided that the boundary $\partial \Omega$ is smooth.
\end{remark}

Associated with problem \eqref{sta3}, the following \L ojasiewicz-Simon type inequality has been proven in \cite[Lemma 4.1]{WGZ07} (see \cite[Lemma 2.2]{Zhang}
for a slightly weaker version).
\begin{lemma}\label{LS1}
   Let \textbf{(F1)} and \textbf{(F2)} be satisfied.
   Suppose that  $\psi$ is a
   critical point of $\Upsilon(\phi)$ in $H^1(\Omega)$. Then there exist constants
   $\theta\in(0,\frac{1}{2})$ and $\beta >0$, depending on  $\psi$, $m$ and $\Omega$, such that,
   for any $\phi \in H^1(\Omega)$ with $\|\phi-\psi\|_{H^1(\Omega)}< \beta$,
   it holds
\begin{align}
   \left\| - \Delta \phi + F'(\phi)-\left(m - |\Omega|^{-1}\int_\Omega
   \phi \dx\right)\right\|_{(H^1(\Omega))'} \ \geq\  | \Upsilon(\phi)
   -\Upsilon(\psi)|^{1-\theta}.
   \label{LSa}
\end{align}
\end{lemma}

Now we are in a position to prove the \L ojasiewicz-Simon type inequality \eqref{LSb} stated in Lemma \ref{LS2}, which plays a crucial role in the study of
long-time behavior of problem \eqref{p1}--\eqref{ini}.\medskip

\textbf{Proof of Lemma \ref{LS2}.}
 In Lemma \ref{LS1}, we take $m=m_\infty$ (see \eqref{m}) and $\psi=\phi_\infty$. Then it follows from \eqref{sta1}--\eqref{stamu1} and \eqref{m} that
 $\phi_\infty$ satisfies the reduced elliptic problem \eqref{sta3}. Hence, according to Lemma \ref{critical}, we see that it is a critical point of
 $\Upsilon(\phi)$ (cf. \eqref{E1} with $m=m_\infty$).
 As a consequence, Lemma \ref{LS1} applies with constants $\theta\in (0, \frac12)$, $\beta>0$ depending
 on  $\phi_\infty$, $m_\infty$ and $\Omega$.
 On the other hand, for any $\phi\in H^2_N(\Omega)$ we set
 $$\mu=-\Delta \phi+F'(\phi)$$
and then using integration by parts, we get
$$\int_\Omega \mu \dx=\int_\Omega F'(\phi) \dx.$$
From the \L ojasiewicz-Simon inequality \eqref{LSa} (applying to $\psi=\phi_\infty)$,
 Poincar\'e's inequality and \eqref{cons}, we deduce that
\begin{align}
& |\Upsilon(\phi) -\Upsilon(\phi_\infty) |^{1-\theta} \nonumber\\
&\quad  \leq
    \left\| - \Delta \phi + F'(\phi)- \left(m_\infty -
   |\Omega|^{-1}\int_\Omega \phi \dx\right) \right\|_{(H^1(\Omega))'}\nonumber\\
&\quad \leq \left\|- \Delta \phi + F'(\phi)-|\Omega|^{-1}\int_\Omega F'(\phi)\dx \right\|_{(H^1(\Omega))'}   \nonumber\\
&\qquad + \left\| |\Omega|^{-1}\int_\Omega F'(\phi)\dx-\left(m_\infty
   -|\Omega|^{-1}\int_\Omega \phi \dx\right)\right\|_{(H^1(\Omega))'}\nonumber\\
 &\quad
 = \left\|\mu- \overline{\mu}  \right\|_{(H^1(\Omega))'}
  + \left\| |\Omega|^{-1}\int_\Omega \mu \dx-|\Omega|^{-1}\int_\Omega \sigma \dx-m_u\right\|_{(H^1(\Omega))'}\nonumber\\
 & \quad \leq \left\|\mu- \overline{\mu}   \right\|_{(H^1(\Omega))'}
  + |\Omega|^{-1}\left|\int_\Omega (\mu - \sigma)\dx\right|+ |m_u|\nonumber\\
  & \quad \leq  \left\|\mu- \overline{\mu}  \right\|_{(H^1(\Omega))'}
  + |\Omega|^{-1} \left(\int_\Omega \frac{1}{P(\phi)} \dx\right)^\frac12\left(\int_\Omega P(\phi)(\mu-\sigma)^2 \dx\right)^\frac12\nonumber\\
  &\qquad +|m_u|.
  \label{LS2A}
   \end{align}
 By the Sobolev embedding $H^2(\Omega)\hookrightarrow C(\overline{\Omega})$ ($n=2,3$), the continuity as well as the strictly positivity of $P(s)$, then it holds
 $$\int_\Omega \frac{1}{P(\phi)} \dx \leq |\Omega|\left(\min_{x\in \overline{\Omega}} P(\phi(x))\right)^{-1}\leq C,$$
 where the constant $C>0$ depends on $\Omega$, $\|\phi\|_{H^2(\Omega)}$ and $P$.

On the other hand, on account of \eqref{E}, \eqref{cons}, \eqref{E1} and Poincar\'e's inequality, since $\theta\in (0,\frac12)$ and $\sigma\in H^1(\Omega)$,
 we infer that
\begin{align}
&|\mathcal{E}(\phi,
\sigma)-\Upsilon(\phi)|^{1-\theta}
&\quad\nonumber\\
&\quad = \left| \frac{1}{2}\|\sigma\|_{L^2(\Omega)}^2-\frac12|\Omega|( \overline{\sigma}+m_u)^2\right|^{1-\theta}\nonumber\\
&\quad \leq \left(\frac{1}{2}\right)^{1-\theta}\left(\int _{\Omega}(\sigma-\overline{\sigma})^2\dx +2|\Omega||\overline{\sigma}||m_u|
+|\Omega|m_u^2\right)^{1-\theta}\nonumber\\
&\quad \leq C\|\nabla
  \sigma\|_{L^2(\Omega)}^{2(1-\theta)}+C\left(|m_u|^{1-\theta}+|m_u|^{2(1-\theta)}\right)\nonumber\\
  &\quad
  \leq C\|\nabla \sigma\|_{L^2(\Omega)}+C|m_u|^\frac12.
  \label{mean}
\end{align}
Finally, since
$\Upsilon(\phi_\infty)=\mathcal{E}(\phi_\infty,\sigma_\infty)$ (recalling that $\sigma_\infty$ is a constant satisfying \eqref{cons}),
we deduce from inequalities \eqref{LS2A} and \eqref{mean} that %
\begin{align}
 & | \mathcal{E}(\phi, \sigma)
   -\mathcal{E}(\phi_\infty, \sigma_\infty)|^{1-\theta} \nonumber\\
  &\quad  \leq
   |\mathcal{E}(\phi, \sigma)-\Upsilon(\phi)|^{1-\theta}+ | \Upsilon(\phi)
   -\Upsilon(\phi_\infty)|^{1-\theta}\nonumber\\
  &\quad \leq
  \left\|\mu- \overline{\mu}  \right\|_{(H^1(\Omega))'}+ C\|\nabla \sigma\|_{L^2(\Omega)}
  + C\| \sqrt{P(\phi)}(\mu-\sigma)\|_{L^2(\Omega)}+C|m_u|^\frac12.\nonumber
\end{align}
The proof of Lemma \ref{LS2} is complete.

\section*{Acknowledgements}

This research has been performed in the framework of the project Fondazione Cariplo-Regione Lombardia MEGAsTAR
``Matema\-tica d'Eccellenza in biologia ed ingegneria come acceleratore
di una nuova strateGia per l'ATtRattivit\`a dell'ateneo pavese''.
C. Cavaterra and E. Rocca were partially supported by GNAMPA (Gruppo Nazionale per l'Analisi Matematica, la Probabilit\`a e le loro Applicazioni)
of INdAM (Istituto Nazionale di Alta Matematica). H. Wu was partially supported by NNSFC grant No. 11631011 and the Shanghai Center for Mathematical Sciences.
This research was also supported by the Italian Ministry of Education, University and Research (MIUR): Dipartimenti di Eccellenza Program (2018--2022) -
Dept. of Mathematics ``F. Casorati'', University of Pavia.


\end{document}